\documentclass[11pt, letterpaper]{article}
\usepackage{amsmath, amsthm, amssymb, amsfonts} 
\usepackage{mathrsfs} 
\usepackage{bm} 
\usepackage{graphicx} 
\usepackage{enumerate} 
\usepackage{geometry} 
\usepackage{setspace} 
\usepackage{lmodern} 
\usepackage{hyperref} 
\usepackage{color} 
\usepackage{xcolor} 
\usepackage{url} 
\usepackage{mathtools}
\usepackage{enumitem}
\usepackage{changepage}
\usepackage{microtype}
\usepackage{authblk}
\usepackage{amsfonts}
\usepackage{mathrsfs,amscd,amssymb,amsthm,amsmath,bm,graphicx,psfrag,subfigure,url,mathtools}
\usepackage{pict2e}
\usepackage{psfrag,amsmath}
\usepackage{tikz}
\usepackage{indentfirst}
\usepackage{hyperref}
\usepackage{bookmark}
\usepackage{enumerate}
\usepackage{latexsym,euscript,epic,eepic,color}
\usepackage{multirow}
\usepackage{multicol}
\usepackage{longtable}
\usepackage{adjustbox}

\usepackage{setspace}
\usepackage{epstopdf}
\allowdisplaybreaks
\usepackage{authblk}
\usepackage{pifont}
\numberwithin{equation}{section}

\usepackage{enumitem}
\usepackage{microtype}

\theoremstyle{plain}
\newtheorem{theorem}{Theorem}[section]
\newtheorem{lemma}[theorem]{Lemma}
\newtheorem{proposition}[theorem]{Proposition}

\newtheorem{problem}{Problem}
\newtheorem{corollary}[theorem]{Corollary}
\theoremstyle{definition}

\theoremstyle{remark}
\newtheorem{remark}[theorem]{Remark}

\hypersetup{
    colorlinks=true,
    linkcolor=blue,
    citecolor=red,
    urlcolor=magenta,
}

\newcommand{\keywords}[1]{%
  \par\vspace{6pt}\noindent\textbf{Keywords: }#1\par
}

\newcommand{\MSC}[2][2020]{%
  \par\vspace{3pt}\noindent\textbf{MSC(#1): }#2\par
}

\title{\bf The edge spectral extremal problem for $kK_3$
in nonzero residue classes}
\vspace{6mm}
\author[1]{Jing Gao}
\author[1]{Shuchao Li}
\author[2]{Yuantian Yu\thanks{Corresponding author. \\
\hspace*{2em}E-mail address: gjing1270@163.com (J. Gao), lscmath@ccnu.edu.cn (S. Li), ytyumath@sina.com (Y. Yu)}}

\affil[1]{School of Mathematics and Statistics, and Hubei Key Lab--Math. Sci.,\linebreak Central China Normal University, Wuhan 430079, China}
\affil[2]{School of Science, East China University of Technology, Nanchang 330013, China}

\date{\today}

\allowdisplaybreaks
\begin{document}
\baselineskip=0.23in

\maketitle

\begin{abstract}
For a fixed integer \(k\ge 2\), let \(kK_3\) denote the vertex-disjoint
union of \(k\) triangles. A recent fixed-size spectral theorem of Das
and Yamini asserts that, for all sufficiently large \(m\), every
\(kK_3\)-free graph \(G\) of size \(m\) satisfies
\(
\lambda(G)\le (k-1)+\sqrt{m-k(k-1)},
\)
and equality holds if and only if $(2k-1)\mid m$ and
\(
G\cong \bigl(K_{2k-1}\vee qK_1\bigr)\cup tK_1,\, q=\frac{m}{2k-1}-(k-1)
\)
for some $t\ge 0$. They explicitly posed the open problem:
Let $k\ge 2$ be fixed and $\ell$ be a residue in $\{1,\dots,2k-2\}$.
For all sufficiently large integers $m\equiv \ell\pmod{2k-1}$, determine the exact value of
\(
\max\bigl\{\lambda(G): e(G)=m,\ G\text{ is }kK_3\text{-free}\bigr\},
\)
and characterize all graphs attaining this maximum. In this paper, using the positive-defect
version of the bounded-core method for divisible sizes together with several new ideas developed in this paper,
we give a complete solution to the aforementioned open problem.
\end{abstract}


\MSC{05C50, 05C35}

\hypertarget{introduction}{%
\section{Introduction}}
All graphs considered in this paper are finite, undirected, and simple.
Let $G=(V(G),E(G))$ be a graph with vertex set $V(G)$ and edge set $E(G)$,
and write $n=|V(G)|$ for its \textit{order} and $m=e(G)=|E(G)|$ for its \textit{size}.

A graph $G$ is said to be $H$-free if $G$ does not contain $H$ as a (not necessarily induced) subgraph.
Given a family of graphs $\mathcal{H}$, $G$ is $\mathcal{H}$-free if it is $H$-free for every $H\in\mathcal{H}$. For a graph \(G\), write \(A(G)\) for
its adjacency matrix, and \(\lambda(G)\) for the spectral radius of \(A(G)\).

A central topic in extremal graph theory is the \textit{Tur\'an type problem}: Determine
\[
\operatorname{ex}(n,F)=\max\bigl\{e(G)\colon |V(G)|=n,\, G\text{ is }F\text{-free}\bigr\},
\]
the maximum number of edges in an $n$-vertex $F$-free graph. The \textit{edge-spectral Tur\'an type problem}, initiated by Nikiforov and others, aims to ask
\begin{equation*}\label{e:1}
\operatorname{spex}(m,F)
=\max\{\lambda(G): G\in \mathcal{G}(m,F)\},
\end{equation*}
where $\mathcal{G}(m,F)=\{G: e(G)=m,\, G\text{ is }F\text{-free}\}.$ It is also known as the Brualdi--Hoffman--Tur\'an type problem; see \cite{brualdi1985spectral}.
In recent years, the above edge spectral Tur\'an type problem has drawn increasing attention from scholars and emerged as one of the frontier research topics in spectral graph theory.

The pioneering work in this field is due to Nosal~\cite{Nosal1970}, he obtained $\lambda(G)\le \sqrt{m}$ if $G$ is triangle-free with $m$ edges. Zhai, Lin, and Shu~\cite{zhai2021} obtained several foundational results. They determined the spectral extrema for $C_5$-free ($m\ge 8$) and $C_6$-free ($m\ge 22$) graphs. They raised an open question: what is the extremal graph among $\mathcal{G}(m,C_{2k+1}) \cup \mathcal{G}(m,C_{2k+2})$? Sun, Li, and Wei~\cite{sun2023} identified the graph among $\mathcal{G}(m,\theta_{1,2,3})$ (resp. $\mathcal{G}(m,\theta_{1,2,4})$) having maximum spectral radius. They further investigated $\{C_3,C_5\}$-free non-bipartite graphs and posed an open problem for even $m$. Li, Sun, and Yu~\cite{li2022dm} considered graphs forbidding all short odd cycles $\{C_3,C_5,\dots,C_{2k+1}\}$; They derived trace-based upper bounds for $\lambda_1^{2k}+\lambda_2^{2k}$ and characterized the corresponding extremal graphs; For non-bipartite $\{C_3,C_5,\dots,C_{2k+1}\}$-free graphs, a spectral-radius inequality was also established. Li, Shen, and Zhao~\cite{li2026spectral} solved the Zhai-Shu open question for even $m\ge 70$ over triangle-free non-bipartite graphs, classifying extremal graphs by $m \pmod 3$. Li, and Yu~\cite{li2023spectral} treated $\{C_3,C_5\}$-free non-bipartite graphs for large even $m$, and proposed a conjecture for general odd-cycle-free non-bipartite graphs with even size. Later this conjecture was resolved by Lou, Lu and Zhai~\cite{lou2025refinement}. Lou, Lu, and Huang~\cite{lou2024spectral} identified the extremal graph for $\{C_3,C_5,\dots,C_{2k+1}\}$-free non-bipartite graphs with odd size $m$.

Li, Zhao, and Zou~\cite{li2025spectral} studied $F_{2k+2}$-free graphs (here $F_k$ is called a fan graph, defined by $F_k=K_1\vee P_{k-1}$). Using a uniform approach, they confirmed the Yu-Li-Peng conjectures on fan graphs and friendship graphs. Chang et al. \cite{liu2026brualdi} determine the graph among $\mathcal{G}(m,\{\theta_{1,2,2k-1},\theta_{1,2,2k}\})$ (resp. $\mathcal{G}(m,\{C_{2k-1},C_{2k}\})$ having maximum spectral radius under the condition $m+\frac{k(k+1)}{2}\not\equiv 0 \pmod{k}$.  Zhang, and Wang~\cite{zhang2026spectral} considered star-forest-free graphs with large $m$ and gave the spectral upper bound together with the extremal graph. Very recently, Fang, Zhai and Zhang~\cite{FangZhaiZhang}, Chen, Gao and Li~\cite{cgl2026}, Gao, Geng and Li~\cite{GaoGengLi} studied $W_{2k+1}$-free graphs (here $W_{2k+1}$ is called a wheel graph, defined by $W_{2k+1}=K_1\vee C_{2k}$), solved a conjecture posed by Yu, Li and Peng~\cite{YLP} and a conjecture posed by Yu, Zhang, and Zhang~\cite{yuzhang2025survey}. For more recent advances along this line, one may consult the nice survey of Yu, Zhang, and Zhang~\cite{yuzhang2025survey}.

Our work in this paper {is} motivated firstly by Wang, Jia, and Ni~\cite{WangJiaNi2026}, who 
showed that if $m \ge 18$ and $G$ is a $2K_3$-free graph with $m$ edges, then
\(
\lambda(G) \le 1+\sqrt{m-2}
\)
with equality if and only if $m\equiv 0\pmod{3}$ and $G\cong \bigl(K_3 \vee \frac{m-3}{3}K_1\bigr)\cup tK_1$ for some $t\ge 0$. Subsequently, Das and Yamini~\cite{DasYamini2026} determined the graph among $\mathcal{G}(m,kK_3)$ having maximum spectral radius for $k\ge 3.$ We describe it as follows.
\begin{theorem}[Das and Yamini~\cite{DasYamini2026}]
\label{Das}
For every integer $k\ge 2$ there is a constant $M_0=M_0(k)$ such that the following holds.
Let $m\ge M_0$ and let $G$ be a $kK_3$-free graph with $e(G)=m$. Then
\(
\lambda(G)\le (k-1)+\sqrt{m-k(k-1)}.
\)
Equality holds if and only if $(2k-1)\mid m$ and
\(
G\cong \bigl(K_{2k-1}\vee qK_1\bigr)\cup tK_1,\, q=\frac{m}{2k-1}-(k-1)
\)
for some $t\ge 0$. In particular, the nontrivial component of every equality graph is unique.
\end{theorem}
Das and Yamini~\cite{DasYamini2026} explicitly posed the following problem.
\begin{problem}[Das and Yamini~\cite{DasYamini2026}]
\label{prob:nondivisible}
Fix $k\ge 2$ and a residue $\ell\in\{1,\dots,2k-2\}$.
For all sufficiently large integers $m\equiv \ell\pmod{2k-1}$, determine the exact value of
\(
\max\bigl\{\lambda(G): e(G)=m,\ G\text{ is }kK_3\text{-free}\bigr\},
\)
and characterize all graphs attaining this maximum.
\end{problem}
In this paper, we address Problem~\ref{prob:nondivisible}, and give a complete resolution for it, which covers the remaining cases of Theorem~\ref{Das}. In order to formulate our result, we define
a graph \(T_{k,m}^{(\ell)}\) on a vertex partition
\(
R\cup U\cup\{z\}
\)
satisfying the following: The graph induced by $R$ is $K_{2k-1};$ \(U\) is an independent set of size \(t=\frac{m-\ell}{2k-1}-(k-1)\), and \(N(u)=R\) for every
  \(u\in U\); \(z\) has exactly \(\ell\) neighbors in \(R\) and no neighbors in \(U\).
The following is our main result in this paper.
\begin{theorem}\label{main-theorem} For every fixed integer \(k\ge2\),
there exists \(M(k)\) such that the following holds. Let \(m\ge M(k)\) with $m\equiv \ell\pmod{2k-1}$, where $1\le\ell\le2k-2$.
If \(G\) is a \(kK_3\)-free graph of size \(m\), then
\(
\lambda(G)\le\lambda\bigl(T_{k,m}^{(\ell)}\bigr).
\)
Equality holds if and only if
\(
G\cong T_{k,m}^{(\ell)}\cup sK_1
\)
for some \(s\ge0\).
\end{theorem}


\noindent{\bf Our approach.}\ The proof develops a positive-defect version of the bounded-core method
used for divisible sizes in~\cite{DasYamini2026}. The first new ingredient is a sharp Perron gap
for cones: the apex of \(K_1\vee H\) carries a largest Perron
coordinate, and \(\lambda(K_1\vee H)-\lambda(H)\ge1\). This permits the cone
case to be excluded without invoking the result being proved. An exact
defect identity then forces a clique core of order \(2k-1\), an independent twin class of linear size complete to the core, and a bounded exceptional
layer. A Schur-complement expansion reduces the remaining problem to a
finite congruence optimization whose nonstandard configurations have a
uniform discrete gap. That gap is larger than the analytic error term
and yields exact uniqueness.\vspace{2mm}

\noindent{\bf Organization.}\ \ In the remainder of this section, we give some necessary notation and terminology.  In Section \ref{preliminaries}, we give some preliminaries. In Section \ref{s3}, {we} investigate some properties of the candidate extremal graph. In Section~\ref{s4}, we give some key structure lemmas for characterization of the extremal graph. In Section~\ref{s5}, we give the proof of Theorem~\ref{main-theorem}.\vspace{2mm}

\noindent{\bf Some notation and definitions.}\
For a vertex $v\in V(G)$, $N_G(v)$ denotes its neighborhood and $d_G(v)=|N_G(v)|$ its degree;
when no confusion arises, we drop the subscript. A vertex is an isolated vertex if its degree is 0. For two graphs $G_1$ and $G_2$, $G_1\cup G_2$ denotes their vertex-disjoint union,
and $G_1\vee G_2$ denotes their \textit{join}, obtained from the disjoint union $G_1\cup G_2$
by adding all edges between $V(G_1)$ and $V(G_2)$. For two sets $A$ and $B$, we write $A\sqcup B$ for the disjoint union of $A$ and $B$.

For a vertex subset $R\subseteq V(G)$, $G[R]$ denotes the subgraph of $G$ induced by $R$,
and we write $e(R):=e(G[R])$ for the number of edges inside $R$.
For two disjoint vertex subsets $R,S\subseteq V(G)$, $e(R,S)$ denotes the number of edges
with one endpoint in $R$ and the other in $S$, and for $v\in V(G)$ we write $d_R(v):=|N_G(v)\cap R|$.
For a graph $G$, a \textit{vertex cover} of $G$ is a subset $K$ of $V(G)$ such that every edge of $G$ has at least one endpoint in $K$. A vertex cover $K^\ast$ is a minimum vertex cover if $G$ has no vertex cover $K$ with $|K|<|K^\ast|$. The number of vertices in a minimum vertex cover of $G$ is called the \textit{vertex covering number} of $G$, and denoted by $\tau(G)$. Two distinct edges in a graph $G$ are independent if they are not incident with a common vertex in $G$. A set of pairwise independent edges in $G$ is called a matching in $G$. A \textit{matching} $M$ is a maximum matching of a graph $G$ if there does not exist a matching $M'$ in $G$ such that $|M'|>|M|$. The \textit{matching number} $\nu(G)$ (or just $\nu$, for short) of $G$ is the cardinality of a maximum matching of $G$.

As usual, we will use $K_n,C_n$ to denote complete graph and cycle on $n$ vertices, respectively.
For a graph $H$ and an integer $k\ge 1$, $kH$ denotes the disjoint union of $k$ vertex-disjoint copies of $H$.

Given a vector $\mathbf{x} \in \mathbb{R}^n$, one has
\(
    \|\mathbf{x}\|_2 = \left(\sum_{i=1}^{n} x_i^2\right)^{1/2}
    =\bigl(\mathbf{x}^{\mathsf T} \mathbf{x}\bigr)^{1/2}
\). For any real {symmetric} matrix $M\in \mathbb{R}^{n\times n}$, its \emph{spectral norm} is
\(
\|M\|_2
=\max_{\|x\|_2=1}\|M\boldsymbol{x}\|_2
=\max\bigl\{|\lambda|\,\big|\) $\lambda$ is an eigenvalue of $M\bigr\}$.
The \emph{Frobenius norm} is defined by
\(
\|M\|_F
=\sqrt{\sum_{i,j} |M_{ij}|^2}.
\)
By \cite[Theorem 5.6.26(b)]{H-J-2012}, one sees $\|A\|_2\leq\|A\|_F$ for any real matrix $A$. A vector $\mathbf{x} \in \mathbb{R}^n$ is called a unit vector if $\|\mathbf{x}\|_2 = 1$. The all-ones vector $\mathbf{1}\in\mathbb{R}^n$ is the column vector all of whose entries equal $1$. For $n \in \mathbb{N}$, let $I_n \in \mathbb{R}^{n\times n}$ denote the $n\times n$ identity matrix, i.e.\ $(I_n)_{ij} =1$ if $i=j$ and $(I_n)_{ij} =0$ if $i\neq j$. Let $J_n\in \mathbb{R}^{n\times n}$ be $n\times n$ \emph{all-ones matrix}, that is, \ $(J_n)_{ij} = 1$ for all $i,j$. Then $J_n=\mathbf{1} \mathbf{1}^{\mathsf T}$.

The \textit{adjacency matrix} $A(G)$ of $G$ is the $n\times n$ $0$-$1$ matrix indexed by $V(G)$ with $A(G)_{uv}=1$ if and only if $uv\in E(G)$.
Since $A(G)$ is real and symmetric, its eigenvalues are real; we denote the largest of them by $\lambda(G)$, called the \textit{spectral radius} of $G$.
When $G$ is connected, $A(G)$ is irreducible and nonnegative, so by the Perron-Frobenius theorem $\lambda(G)$ is a simple eigenvalue admitting a positive eigenvector $\mathbf{x}=(x_v)_{v\in V(G)}$, unique up to positive scaling, called the \textit{Perron vector} of $G$. For a eigenvector $\mathbf{y}$ of $A(G)$ and a subset $V'\subseteq V(G),$ the restriction of $\mathbf{y}$ on $V'$ is denoted by $\mathbf{y}|_{V'}$ or $\mathbf{y}|_{G[V']}$.

Throughout the paper, \(k\) is fixed. {Let $f,g:\mathbb{N}\rightarrow\mathbb{R}$ be functions such that $g(m)>0$ for all sufficiently large integers $m$. We write $f(m)=O_k(g(m))$ if there exist constants \(C_k>0\) and \(m_0=m_0(k)\), depending only on \(k\), such that $|f(m)|\leq C_k g(m)$ for all \(m\geq m_0\). Similarly, we write $f(m)=\Omega_k(g(m))$ if there exist constants \(c_k>0\) and \(m_0=m_0(k)\), depending only on \(k\), such that $f(m)\geq c_k g(m)$ for all \(m\geq m_0\). Thus, every implicit constant in \(O_k(\cdot)\), \(\Omega_k(\cdot)\), and every
sufficiently large threshold may depend on \(k\), but not on \(m\) or on
the residue \(\ell\). If $g(m)\neq0$, we write $f(m)=o(g(m))$ if for every $\varepsilon>0$ there exists $m_0=m_0(\varepsilon)$ such that $|f(m)|\leq\varepsilon g(m)$ for all $m\geq m_0$.}

\section{Preliminaries}\label{preliminaries}

\begin{lemma}[\cite{WuXiaoHong2005}]\label{edge-shifting}
Let \(G\) be connected and let \(\mathbf{x}\)
be a positive Perron vector of \(G\). Suppose that \(uv\in E(G)\),
\(uw\notin E(G)\), and \(x_w\ge x_v\). If \(G'=G-uv+uw\), then $\lambda(G')>\lambda(G)$.
\end{lemma}

\begin{lemma}[\cite{Brouwer-Haemers}]\label{subgraph-spectral-radius}
If $H$ is a subgraph of a connected graph $G$, then $\lambda(H)\leq\lambda(G)$ with equality if and only if $G\cong H$.
\end{lemma}

\begin{theorem}[\cite{ErdosGallai1959}]\label{ErdosGallai}
If \(H\) is an \(n\)-vertex graph
with \(\nu(H)\le s\), then

\[
e(H)\le
\max\left\{\binom{2s+1}{2},
\binom s2+s(n-s)\right\}.
\]
\end{theorem}
By Theorem \ref{ErdosGallai}, for $r\geq1$, if
\(|H|=h\ge2r+1\) and \(\nu(H)\le r\), then
\begin{align}\label{boundofnoperfectmatching}
\frac{e(H)}h\le r.
\end{align}
Equality in \eqref{boundofnoperfectmatching} is possible only when \(h=2r+1\) and
\(H\cong K_{2r+1}\).

\begin{theorem}[\cite{Konig1931}]\label{Konig}
If $G$ is bipartite, then $\tau(G)=\nu(G)$.
\end{theorem}

\begin{theorem}[\cite{DasYamini2026}]\label{Das--Yamini}
Fix \(s\ge2\). If \(H\) is
an \(sK_3\)-free graph of sufficiently large size \(m\geq M_0(s)\), then
\(
\lambda(H)\le (s-1)+\sqrt{m-s(s-1)}.
\)
Consequently, if \(\mu=\lambda(H)\), then
\(
\mu^2-2(s-1)\mu
\le m-(s-1)(2s-1).
\)
\end{theorem}

\begin{lemma}[\cite{Nosal1970}]\label{triangle-free-bound}
If \(H\) is $K_3$-free and has $m$ edges, then $\lambda(H)\le \sqrt{m}$.
\end{lemma}

\section{The graph $T_{k,m}^{(\ell)}$}\label{s3}
From now on, let $k\ge2$ be a fixed integer, 
and let $m$ be a sufficiently large integer with $m\equiv \ell\pmod{2k-1}$, where $1\le\ell\le 2k-2$. Define \(t=\frac{m-\ell}{2k-1}-(k-1)\). By definition, the graph $T_{k,m}^{(\ell)}$ has exactly \(m\) edges. Because $U\sqcup\{z\}$ is an independent set, every triangle of $T_{k,m}^{(\ell)}$ uses at
least two vertices of \(R\). Since $|R|=2k-1$, it follows that
\(T_{k,m}^{(\ell)}\) is \(kK_3\)-free.

Note that $T_{k,m}^{(\ell)}$ has the equitable quotient matrix
\[
Q_\ell=
\begin{pmatrix}
\ell-1&2k-1-\ell&t&1\\
\ell&2k-2-\ell&t&0\\
\ell&2k-1-\ell&0&0\\
\ell&0&0&0
\end{pmatrix}.
\]
By a simple calculation, the characteristic polynomial of $Q_\ell$ is equal to
\begin{align}\notag
p(x)
&=x^4-(2k-3)x^3-((2k-1)t+\ell+2k-2)x^2 \\ \label{polynomialofextremalgraph}
&-\bigl((2k-1)t-\ell(2k-2-\ell)\bigr)x
+\ell(2k-1-\ell)t.
\end{align}
Then $\hat{\lambda}:=\lambda\bigl(T_{k,m}^{(\ell)}\bigr)$ is the largest root of \(p(x)\).

\begin{lemma}\label{candidate-graph}
The spectral radius $\hat{\lambda}$ of graph $T_{k,m}^{(\ell)}$ satisfies
\begin{align}\label{candidate-graph-equation}
\hat{\lambda}^2-2(k-1)\hat{\lambda}
=m-\binom {2k-1}2-\delta_\ell,
\end{align}
where
\begin{align}\label{delta}
\delta_\ell=\ell(2k-1-\ell)
\frac{\hat{\lambda}+t}
{\hat{\lambda}(\hat{\lambda}+1)}.
\end{align}
Moreover,
\begin{align}\label{boundofdelta}
0<\delta_\ell
<D_\ell:=\frac{\ell(2k-1-\ell)}{2k-1}
\le \frac{k(k-1)}{2k-1}
<2k-3.
\end{align}
\end{lemma}
\begin{proof}
Let $\mathbf{y}$ be the positive eigenvector of $A(T_{k,m}^{(\ell)})$ corresponding to $\hat{\lambda}$. Denote $S=N_R(z)$. By symmetry, $\mathbf{y}$ takes the same value on the vertices of
\(S\), \(R\setminus S\) and \(U\), respectively. We may assume that the coordinates of $\mathbf{y}$ corresponding to the vertices of \(S\), \(R\setminus S\), \(U\) and \(\{z\}\) are $y_s,y_r,y_u$ and $y_z$, respectively. By considering the eigenvalue-eigenvector equation of $\hat{\lambda}$ and $\mathbf{y}$ at the vertices of \(S\), \(R\setminus S\), \(U\) and \(\{z\}\), respectively, we have
\begin{align}\label{eigenvalue-eigenvector-equation}
\begin{cases}
\hat{\lambda} y_s=(\ell-1)y_s+(2k-1-\ell)y_r+ty_u+y_z,\\
\hat{\lambda} y_r=\ell y_s+(2k-1-\ell-1)y_r+ty_u,\\
\hat{\lambda} y_u=\ell y_s+(2k-1-\ell)y_r,\\
\hat{\lambda} y_z=\ell y_s.
\end{cases}
\end{align}
The first two equations of (\ref{eigenvalue-eigenvector-equation}) yield
\[
(\hat{\lambda}+1)(y_s-y_r)=y_z.
\]
Combining the above equation and the last equation of (\ref{eigenvalue-eigenvector-equation}), we deduce that
\begin{align}\label{yr}
y_r=(1-\frac{\ell}{\hat{\lambda}(\hat{\lambda}+1)})y_s.
\end{align}
Substituting the last two equations of (\ref{eigenvalue-eigenvector-equation}) and (\ref{yr}) into
the first equation of (\ref{eigenvalue-eigenvector-equation}) and then multiplying both sides by \(\hat{\lambda}\), we have
\[
\hat{\lambda}^2-(2k-2)\hat{\lambda}-(2k-1)t-\ell
=-\ell(2k-1-\ell)
\frac{\hat{\lambda}+t}{\hat{\lambda}(\hat{\lambda}+1)}.
\]
By \(m-\binom {2k-1}2=(2k-1)t+\ell\), \eqref{candidate-graph-equation}--\eqref{delta} hold.

Since $K_{2k-1}\vee tK_1$ is a proper subgraph of $T_{k,m}^{(\ell)}$, by Lemma \ref{subgraph-spectral-radius}, $\hat{\lambda}>\lambda(K_{2k-1}\vee tK_1)$. Note that $\lambda(K_{2k-1}\vee tK_1)$ is the largest root of $x^2-({2k-2})x-(2k-1)t=0$. Therefore,
\(
\hat{\lambda}^2-(2k-2)\hat{\lambda}-(2k-1)t>0.
\)
That is,
\(
(2k-1)(\hat{\lambda}+t)
<\hat{\lambda}(\hat{\lambda}+1).
\)
Hence \(\delta_\ell<D_\ell\). Because $\ell$ is an integer, the maximum of \(\ell(2k-1-\ell)\) on \(1\le\ell\le 2k-2\) is \(k(k-1)\), attained
at \(\ell=k-1,k\), and $k(k-1)<(2k-1)(2k-3)$. This completes the proof.
\end{proof}

By equation \eqref{candidate-graph-equation}, we have $\hat{\lambda}=k-1+\sqrt{m-k(k-1)-\delta_\ell}$. Since $\delta_\ell<2k-3$ and $\sqrt{x-y}-\sqrt{x}=\frac{-y}{\sqrt{x-y}+\sqrt{x}}$ for $x\geq y>0$, we obtain
\begin{align}\label{hatlambda}
\hat{\lambda}=\sqrt m+k-1+O_k(m^{-1/2}).
\end{align}
Because $m=\binom {2k-1}2+(2k-1)t+\ell$, it follows that $t=\frac{m-\binom {2k-1}2-\ell}{2k-1}=\frac{m}{2k-1}+O_k(1)$. Therefore $\hat{\lambda}+t=\frac{m}{2k-1}+\sqrt{m}+O_k(1)$. Also, by equation (\ref{hatlambda}), we get $\hat{\lambda}(\hat{\lambda}+1)=m+(2k-1)\sqrt{m}+O_k(1)$. Hence $\frac{\hat{\lambda}+t}{\hat{\lambda}(\hat{\lambda}+1)}
=\frac{1}{2k-1}+O_k(m^{-1})$. Then
\begin{align}\label{eq:3.1}
\delta_\ell=\ell(2k-1-\ell)(\frac{1}{2k-1}+O_k(m^{-1}))=D_\ell+O_k(m^{-1}).
\end{align}

\begin{remark}
For $k=2$,  we have $\ell(3-\ell)=2$. 
By equation \eqref{delta} and inequality \eqref{boundofdelta}, we have
\begin{align}\label{deltar2}
 \delta_\ell
 =2\frac{\hat\lambda+t}{\hat\lambda(\hat\lambda+1)},
 \qquad 0<\delta_\ell<\frac23<1.
\end{align}
\end{remark}

\section{Preliminary characterization for the extremal graph}\label{s4}
In this section, we obtain a rough characterization of the extremal graph structures. For this purpose, we use the following lemmas.
\begin{lemma}\label{maximum-coordinateapex}
Let $H$ be a graph with at least one vertex and \(\hat{H}=\{u\}\vee H\). Suppose that $\mathbf{y}$ is the positive eigenvector corresponding to $\lambda(\hat{H})$. Then  the coordinate of $\mathbf{y}$ corresponding to $u$ is the maximum in $\mathbf{y}$ and $\lambda(\hat{H})-\lambda(H)\geq1$.
\end{lemma}
\begin{proof}
Let \(y_v\) denote the coordinate of \(\mathbf y\) corresponding to \(v\in V(\hat{H})\).  The eigenvalue
equations at \(u\) and at \(v\in V(H)\) are
\begin{align*}
\begin{cases}
 \lambda(\hat{H}) y_u=\sum_{w\in V(H)}y_w,\\[5pt]
 \lambda(\hat{H}) y_v=y_u+\sum_{w\in N_H(v)}y_w.
\end{cases}
\end{align*}
Then
\begin{align*}
 \lambda(\hat{H})(y_u-y_v)
 &=\sum_{w\in V(H)}y_w-y_u-\sum_{w\in N_H(v)}y_w=y_v+\sum_{w\in V(H)\setminus N_H[v]}y_w-y_u.
\end{align*}
Consequently,
\(
 (\lambda(\hat{H})+1)(y_u-y_v)
 =\sum_{w\in V(H)\setminus N_H[v]}y_w\ge0.
\)
Thus \(y_u\ge y_v\) for every \(v\in V(H)\).

Note that
\[
(\lambda(\hat{H}) I_{|V(H)|}-A(H))\mathbf{y}_H=y_u\mathbf1.
\]
Let \(\mathbf{z}\) be a nonzero nonnegative eigenvector of $A(H)$ corresponding to \(\lambda(H)\). Left-multiplying the above equation by \(\mathbf{z}^{\mathsf T}\) on both sides, we derive that
\[
(\lambda(\hat{H})-\lambda(H))\mathbf{z}^{\mathsf T} \mathbf{y}_H
=y_u \mathbf{z}^{\mathsf T}\mathbf1.
\]
Since \(y_u\ge y_v\) for every \(v\in V(H)\), we have \(0<\mathbf{y}_H\le y_u\mathbf1\). Therefore \(\mathbf{z}^{\mathsf T} \mathbf{y}_H\le y_u \mathbf{z}^{\mathsf T}\mathbf1\). So \(\lambda(\hat{H})-\lambda(H)\ge1\). \end{proof}

\begin{lemma}\label{strengthened-cone-exclusion}
Let \(H\) be \((k-1)K_3\)-free and \(\hat{H}=K_1\vee H\). If
\(e(\hat{H})=m\), then $\lambda(\hat{H})<\hat{\lambda}$.
\end{lemma}
\begin{proof}
Denote $n=|V(H)|$, $e=e(H)$. Then $m=n+e$.
Write \(\rho=\lambda(\hat{H})\) and \(\mu=\lambda(H)\). By Lemma \ref{maximum-coordinateapex}, we have $\rho>\mu$. Let $\mathbf{y}$ be the positive eigenvector corresponding to $\lambda(\hat{H})$, 
Then
\begin{align*}
\begin{cases}
\mathbf1^{\mathsf T}\mathbf{y}_H=\rho y_1,\\
\mathbf1 y_1+A(H)\mathbf{y}_H=\rho \mathbf{y}_H,
\end{cases}
\end{align*}
where $y_1$ is the coordinate of $\mathbf{y}$ corresponding to the vertex in $V(\hat{H})\setminus V(H)$. Since $\rho>\mu$, it follows that
\begin{align}\label{rhoequation}
\rho=\mathbf1^\mathsf{T}(\rho I_n-A(H))^{-1}\mathbf1.
\end{align}

Let
\(\theta_1,\ldots,\theta_n\) be the eigenvalues of \(A(H)\), counted with
algebraic multiplicities, and choose an orthonormal eigenbasis
\(\mathbf{v}_1,\ldots,\mathbf{v}_n\) of \(\mathbb{R}^n\) such that for every $i\in [n]$,
$$
\text{ $A(H)\mathbf{v}_i=\theta_i \mathbf{v}_i$\qquad  and \qquad $
\mathbf{v}_i^{\mathsf{T}}\mathbf{v}_j=
\begin{cases}
1, &i=j,\\
0, &i\neq j.
\end{cases}
$}
$$
Write $M=(\mathbf{v}_1,\cdots,\mathbf{v}_n)$, and $\Lambda=diag(\theta_1,\cdots,\theta_n)$. Then $A(H)=M\Lambda M^\mathsf{T}=\sum_{i=1}^n\theta_i\mathbf{v}_i\mathbf{v}_i^\mathsf{T}$. For each \(i\in[n]\), set $c_i=(\mathbf{1}^\mathsf{T}\mathbf{v}_i)^2$. Then $\sum_{i=1}^n c_i=n$  and $\sum_{i=1}^n\theta_i c_i=\sum_i\theta_i \mathbf{1}^\mathsf{T}\mathbf{v}_i\mathbf{v}_i^\mathsf{T}\mathbf1=\mathbf1^\mathsf{T} (A(H))\mathbf1=2e$. Moreover, $(\rho I_n-A(H))^{-1}=\sum_{i=1}^n\frac{1}{\rho-\theta_i}\mathbf{v}_i\mathbf{v}_i^\mathsf{T}$. Hence,
by equation (\ref{rhoequation}) and \(|\theta_i|\le\mu<\rho\), we have
\begin{align*}
\rho&=\sum_{i=1}^n\frac{(\mathbf{1}^\mathsf{T}\mathbf{v}_i)^2}{\rho-\theta_i}
=\sum_{i=1}^n(\mathbf{1}^\mathsf{T}\mathbf{v}_i)^2\frac{\rho+\theta_i}{\rho^2-\theta_i^2}
\leq\sum_{i=1}^n\frac{c_i(\rho+\theta_i)}{\rho^2-\mu^2}\
=\frac{n\rho+2e}{\rho^2-\mu^2}.
\end{align*}
It follows that
\begin{align}\label{rhomu}
\rho^3\le\rho(n+\mu^2)+2e.
\end{align}

Suppose first that \(k=2\). Then \(H\) is $K_3$-free. By Lemma~\ref{triangle-free-bound}, we obtain \(\mu^2\le e\). Since \(m=n+e\), by inequality \eqref{rhomu}, we get $\rho^3\le m\rho+2m$. Let
$Q_{1,m}(x):=x^3-mx-2m$. Then $Q_{1,m}(\rho)\leq0$. On the other hand, by equation \eqref{candidate-graph-equation}, we have
\(
 m=\hat\lambda^2-2\hat\lambda+3+\delta_\ell.
\)
Therefore,
\(
 Q_{1,m}(\hat\lambda)=\hat\lambda^3-m\hat\lambda-2m
 =(1-\delta_\ell)\hat\lambda-6-2\delta_\ell.
\)
By equation \eqref{deltar2}, \(1-\delta_\ell>1/3\), and hence $Q_{1,m}(\hat\lambda)>0$ for sufficiently large \(m\).  Furthermore,
\[
 Q_{1,m}'(\hat\lambda)
 =3\hat\lambda^2-m
 =2\hat\lambda^2+2\hat\lambda-3-\delta_\ell>0
\]
for sufficiently large \(m\), while \(Q_{1,m}''(x)=6x>0\) for
\(x>0\).  Thus \(Q_{1,m}(x)\) is strictly increasing on
\([\hat\lambda,\infty)\).  If \(\rho\ge\hat\lambda\), then
\(
 Q_{1,m}(\rho)\ge Q_{1,m}(\hat\lambda)>0,
\)
a contradiction with $Q_{1,m}(\rho)\leq0$. So $\rho=\lambda(\hat{H})<\hat\lambda$.

We may therefore assume that \(k\ge3\). Let $M_0(k)$ be a constant satisfying Theorem \ref{Das}.
If \(e\leq M_0(k)\), then equation \eqref{rhomu} gives
\(\rho^2\le m+O_k(1)\). By equations (\ref{candidate-graph-equation}) and (\ref{hatlambda}), we have
\(\hat{\lambda}^2=2(k-1)(\sqrt m+k-1+O_k(m^{-1/2}))+m-\binom {2k-1}2-\delta_\ell=m+2(k-1)\sqrt m+O_k(1)\). Hence $\rho=\lambda(\hat{H})<\hat\lambda$. We may therefore assume that \(e>M_0(k)\). By Theorem \ref{Das--Yamini}, we obtain
\(
n+\mu^2
\le m+2(k-2)\mu-(k-2)(2k-3).
\)
By Lemma \ref{maximum-coordinateapex}, we get \(\mu\le\rho-1\), and hence
\(
n+\mu^2
\le m+2(k-2)\rho-(k-2)(2k-1).
\)
Substituting the above inequality into \eqref{rhomu}, and by
\(e\le m\), we have $\rho^3-2(k-2)\rho^2-m\rho+(k-2)(2k-1)\rho-2m\leq0$.
Let $Q_{k,m}(x)
=x^3-2(k-2)x^2-mx
+(k-2)(2k-1)x-2m$. It follows that $Q_{k,m}(\rho)\le0$.
From equation (\ref{candidate-graph-equation}), we have $m=\hat{\lambda}^2-2(k-1)\hat{\lambda}+(k-1)(2k-1)+\delta_\ell$.
Therefore,
\[
Q_{k,m}(\hat{\lambda})
=(2k-3-\delta_\ell)\hat{\lambda}
-2(k-1)(2k-1)-2\delta_\ell.
\]
By inequality \eqref{boundofdelta}, \(2k-3-\delta_\ell>0\), so $Q_{k,m}(\hat{\lambda})>0$ for large \(m\). Moreover, because $\hat{\lambda}\geq k-1$ and $\delta_\ell<2k-3$, we have
\[
Q_{k,m}'(\hat{\lambda})
=2\hat{\lambda}^2+(6-2k)\hat{\lambda}
-(2k-1)-\delta_\ell>0,
\]
and \(Q_{k,m}''(x)=6x-4(k-2)>0\) for \(x\ge\hat{\lambda}\). Hence
\(Q_{k,m}(x)\) is strictly increasing on \([\hat{\lambda},\infty)\). If
\(\rho\ge\hat{\lambda}\), then
\(
Q_{k,m}(\rho)
\ge Q_{k,m}(\hat{\lambda})>0,
\)
a contradiction with $Q_{k,m}(\rho)\leq0$. So \(\hat{\lambda}>\rho\).\end{proof}

In the remainder of the paper, we always assume that $G$ is a \(kK_3\)-free graph of size \(m\) having maximum spectral radius. Since isolated vertices do not have an effect on the spectral radius, we consider graph $G$ without isolated vertices.
Now we give some characterization for $G$.
\begin{proposition}\label{connected}
The graph \(G\) is connected.
\end{proposition}
\begin{proof}Suppose otherwise, and let \(G_0\) be a component with
\(\lambda(G_0)=\lambda(G)\). Let the other components contain altogether
\(e_1>0\) edges, and choose \(u\in V(G_0)\). Add \(e_1\) new vertices $v_1,\ldots,v_{e_1}$. Let \(G'=G_0\cup\{uv_1,\ldots,uv_{e_1}\}\). Then \(G'\) is \(kK_3\)-free with $m$ edges.

Let $\mathbf{z}$ be the unit Perron vector of \(G_0\). Let $\tilde{\mathbf{z}}$ be the vector of dimension $|V(G')|$ satisfying $\tilde{\mathbf{z}}_v=\mathbf{z}_v$ for $v\in V(G_0)$ and $\tilde{\mathbf{z}}_{v_i}=0$ for $i\in[e_1]$. Then $\tilde{\mathbf{z}}^\mathsf{T}\tilde{\mathbf{z}}=1$ and $\tilde{\mathbf{z}}^\mathsf{T}A(G')\tilde{\mathbf{z}}=\mathbf{z}^\mathsf{T}A(G_0)\mathbf{z}$. Therefore $\lambda(G')\geq\tilde{\mathbf{z}}^\mathsf{T}A(G')\tilde{\mathbf{z}}=\mathbf{z}^\mathsf{T}A(G_0)\mathbf{z}=\lambda(G_0)$.
Since \(\lambda(G_0)=\lambda(G)\), we have \(\lambda(G')\ge\lambda(G)\). If $\lambda(G')=\lambda(G)=\lambda(G_0)$, we have $A(G')\tilde{\mathbf{z}}=\lambda(G')\tilde{\mathbf{z}}$. It follows that $0=\lambda(G')\tilde{\mathbf{z}}_{v_1}=(A(G')\tilde{\mathbf{z}})_{v_1}=\tilde{\mathbf{z}}_u=\mathbf{z}_u>0$, a contradiction. Thus \(\lambda(G')>\lambda(G)\), which contradicts the maximality of $G$. This completes the proof.
\end{proof}

Write $\lambda=\lambda(G)$. By Proposition \ref{connected} and Perron-Frobenius theorem, there exists a positive eigenvector \(\mathbf{x}\) corresponding to $\lambda$ with coordinate $x_v$ corresponding to the vertex $v\in V (G)$. Choose a vertex \(u^\ast\)
at which \(x_{u^\ast}=\max_{v\in V(G)}x_v\), and assume \(x_{u^\ast}=1\). Thus \(0<x_v\le1\) for every vertex $v\in V(G)$ and by the equation at \(u^\ast\), $\lambda=\lambda x_{u^\ast}=\sum_{v\in N_G(u^\ast)}x_v$. In the remainder of the paper, we always assume that $A=N_G(u^\ast)$, $B=V(G)\setminus N_G[u^\ast]$.
\begin{proposition}
The graph \(G-u^\ast\) contains \(k-1\) pairwise
vertex-disjoint triangles.
\end{proposition}
\begin{proof} Suppose that \(G-u^\ast\) is \((k-1)K_3\)-free. If
\(B\ne\emptyset\), take \(w\in B\) and a neighbor \(v\) of \(w\).
Let \(G'=G-wv+wu^\ast\). We claim that \(G'\) is still \(kK_3\)-free.
Otherwise \(G'\) contains \(k\) disjoint triangles. Since \(G\) is \(kK_3\)-free and the only new
edge is \(wu^\ast\), one of these triangles of $G'$ must contain $wu^\ast$.
Removing this triangle from $G'$ leaves \(k-1\) disjoint triangles that use no
\(u^\ast\) and no $wu^\ast$. That is, all of them are in \(G-u^\ast\). This
contradicts the assumption. Hence \(G'\) is \(kK_3\)-free. Since \(x_{u^\ast}\ge x_v\), by Lemma \ref{edge-shifting}, we have $\lambda(G')>\lambda(G)$, a contradiction. Therefore, \(B=\emptyset\). {That is, \(G=\{u^\ast\}\vee(G-u^\ast)\). 
Note that \(G-u^\ast\) is \((k-1)K_3\)-free and $e(G)=m$ is sufficiently large.}
By Lemma \ref{strengthened-cone-exclusion}, we obtain \(\lambda<\hat{\lambda}\), contradiction with the maximality of $G$. This completes the proof.
\end{proof}

Fix disjoint triangles \(T_1,\ldots,T_{k-1}\) in \(G-u^\ast\) and set $T=\bigcup_{i=1}^{k-1}V(T_i)$,
and $S_0=A\cap T$. Then the set \(S_0\) is a vertex cover of \(G[A]\). Indeed, if
\(xy\in E(G[A])\) and \(\{x,y\}\cap S_0=\emptyset\), then the triangle
\(u^\ast xyu^\ast\) is vertex-disjoint from \(T_1,\ldots,T_{k-1}\), and so there are \(k\)
disjoint triangles in $G$, a contradiction. In particular, $|S_0|\le3(k-1)$.

Because \(T_{k,m}^{(\ell)}\) is $kK_3$-free, we have $\lambda\ge\hat{\lambda}$. For every graph \(H\) of size \(m\), define
\[
\mathfrak D_k(H)
=m-(2k-1)(k-1)+2(k-1)\lambda(H)-\lambda^2(H).
\]
The function \(f(x)= -x^2+2(k-1)x+m-(2k-1)(k-1)\) is decreasing for \(x>k-1\).
From equation (\ref{candidate-graph-equation}),
\begin{align}\label{bound DkG}
\mathfrak D_k(G)\le\mathfrak D_k(T_{k,m}^{(\ell)})=m-(2k-1)(k-1)+2(k-1)\hat{\lambda}-\hat{\lambda}^2=\delta_\ell.
\end{align}

Write $C=\{v\in A:d_A(v)>2(k-1)\}$ and $L=A\setminus C$. Let $Q$ be the set of endpoints of the edges of \(G[L]\).

\begin{lemma}\label{exact-defect-identity}
It holds that $2(k-1)\le|C|\le4(k-1)$.
Furthermore, write $|C|=2(k-1)+q$. Then
\[
\mathfrak D_k(G)+\binom q2=\binom{|C|}{2}-e(C)
+e(L)+e(B)
+\mathcal E_C+\mathcal E_L+\mathcal E_B,
\]
where
\[
\mathcal E_C
=\sum_{v\in C}(d_A(v)-2k+2))(1-x_v)\geq0,\
\mathcal E_L
=\sum_{v\in L}(2k-2-d_A(v))x_v\geq0,\
\mathcal E_B
=\sum_{v\in B}d_A(v)(1-x_v)\geq0.
\]
\end{lemma}
\begin{proof} We first prove \(|C|\ge2(k-1)\). Let $N_0=\binom{3k-3}{2}+1$
and $A_1=\{v\in S_0:d_A(v)\ge N_0\}$. Put \(A_0=A\setminus A_1\). Since \(S_0\)
covers \(E(G[A])\), every edge of \(G[A_0]\) has an endpoint in
\(S_0\setminus A_1\). Because $d_A(v)\leq N_0-1$ for every $v\in S_0\setminus A_1$ and \(|S_0\setminus A_1|\leq|S_0|\le3(k-1)\), $e(G[A_0])\leq(N_0-1)|S_0\setminus A_1|=O_k(1)$.

By $A^2(G)\mathbf{x}=\lambda^2\mathbf{x}$ and $x_{u^\ast}=1$, we derive
\begin{align}\label{lambda2}
\lambda^2
=|A|+\sum_{v\in A}d_A(v)x_v
+\sum_{u\in B}d_A(u)x_u.
\end{align}
For \(v\in A_1\), we have $\lambda x_v\geq\sum_{w\in N_A(v)}x_w+1$. Then $\sum_{w\in N_A(v)}x_w\le\lambda x_v-1\leq\lambda-1$.
Also, by $|A_1|\leq|S_0|\leq3(k-1)$, we deduce
\begin{align*}
\sum_{v\in A_1}d_A(v)x_v&\le\sum_{v\in A_1}d_A(v)\leq e(A)+e(A_1)\leq e(A)+{|A_1|\choose2}
\le e(A)+O_k(1).
\end{align*}
Therefore,
\begin{align*}
\sum_{v\in A}d_A(v)x_v
&=\sum_{v\in A_1}d_A(v)x_v+\sum_{v\in A_0}d_A(v)x_v\\
&\le \sum_{v\in A_1}\left(
d_A(v)x_v+\sum_{w\in N_A(v)}x_w
\right)+2e(G[A_0])\\
&\leq e(A)+O_k(1)+(\lambda-1)|A_1|+2e(G[A_0])\\
&=e(A)+(\lambda-1)|A_1|+O_k(1).
\end{align*}
Note that $\sum_{u\in B}d_A(u)x_u\le e(A,B)$. Hence,
$
\lambda^2\leq|A|+e(A)+(\lambda-1)|A_1|+O_k(1)+e(A,B).
$
By $m=|A|+e(A)+e(A,B)+e(B)$, we obtain
\begin{align}\label{boundlambda2}
\lambda^2-|A_1|\lambda\le m+O_k(1).
\end{align}
Recall that $\mathfrak{D}_k(G)\leq\delta_{\ell}$. It follows that
\(
\lambda^2-2(k-1)\lambda
\ge m-(k-1)(2k-1)-\delta_\ell.
\)
If \(|A_1|\le2k-3\), then the above inequality implies
\(
\lambda^2-|A_1|\lambda
\ge m-(k-1)(2k-1)-\delta_\ell+\lambda,
\)
a contradiction with inequality \eqref{boundlambda2} as \(\lambda\to\infty\). Hence \(|A_1|\ge2(k-1)\).
Since \(N_0>2(k-1)\), we have \(A_1\subseteq C\). Therefore \(|C|\ge2(k-1)\).

If \(c\in C\setminus T\), then \(c\notin S_0\). Recall that
\(S_0\) covers \(G[A]\). It follows that \(N_A(c)\subseteq S_0\). Because
\(d_A(c)>2(k-1)\), by the pigeonhole principle, there is a $T_i$ ($1\leq i\leq k-1$) such that \(c\) is adjacent to all three vertices $t_1,t_2,t_3$ of \(T_i\). If there is another
vertex \(c'\in C\setminus T\) such that \(c'\) is adjacent to all three vertices of \(T_i\), then $ct_1t_2c$, $u^\ast c't_3u^\ast$ and $\{T_j: j\in[k-1]\setminus \{i\}\}$ are \(k\) pairwise vertex-disjoint triangles. Thus for every $T_i$ ($1\leq i\leq k-1$), there is at most one vertex in $C\setminus T$ which is adjacent to all three vertices of \(T_i\). So \(|C\setminus T|\le k-1\). Since $|C\cap T|\leq|T|=3(k-1)$, we have \(|C|\le4(k-1)\).

By $m=|A|+e(A)+e(A,B)+e(B)$, $\lambda=\sum_{v\in A}x_v$ and equation \eqref{lambda2}, we have
\begin{align*}
\mathfrak D_k(G)
&=e(A)+e(B)-(k-1)(2k-1)
+\sum_{v\in A}(2(k-1)-d_A(v))x_v
+\sum_{u\in B}d_A(u)(1-x_u)\\
&=\sum_{v\in C}d_A(v)-e(C)+e(L)+e(B)-(k-1)(2k-1)+\sum_{v\in C}(2(k-1)-d_A(v))x_v\\
&\quad +\sum_{v\in L}(2(k-1)-d_A(v))x_v +\sum_{u\in B}d_A(u)(1-x_u)\\
&=\sum_{v\in C}(d_A(v)-2(k-1))(1-x_v)+2(k-1)|C|-e(C)+e(L)+e(B) \\
&\quad-(k-1)(2k-1)+\mathcal{E}_L+\mathcal{E}_B\\
&=2(k-1)|C|-e(C)+e(L)+e(B)-(k-1)(2k-1)+\mathcal{E}_C+\mathcal{E}_L+\mathcal{E}_B\\
&=\binom{|C|}{2}-\binom{q}{2}-e(C)+e(L)+e(B)+\mathcal{E}_C+\mathcal{E}_L+\mathcal{E}_B,
\end{align*}
as desired. By the definition of $C, L$ and $0<x_v\leq1$ for every $v\in V(G)$, we have $\mathcal{E}_C,\mathcal{E}_L,\mathcal{E}_B\geq0$. 
\end{proof}

By Lemma \ref{exact-defect-identity}, we have \(2(k-1)\leq |C|\leq4(k-1)\). Let $|C|=2(k-1)+q$ where $q=|C|-2(k-1)\in\{0,\ldots,2(k-1)\}$. Write $D_q=\binom q2+\delta_\ell$. Recall that $\delta_\ell<2k-3$. Then $D_q=O_k(1)$.

\begin{lemma}
\label{weighted-neighborhood-class}
\begin{enumerate}[label=\textup{(\roman*)}]
\item $e(L),\ e(B),\ \mathcal E_B\leq D_q$ and $|Q|\leq2D_q$.
\item $\sum_{\substack{z\in L\\d_A(z)<2(k-1)}}x_z\leq D_q$.
\item Let $Z_0=\{z\in L\setminus Q:d_A(z)=2(k-1)\}$. Then $\sum_{z\in Z_0}x_z\geq\lambda-|C|-3D_q$. Moreover, 
    there is a set \(W\subseteq C\) of order \(2(k-1)\) such that $\sum_{z\in Z}x_z\geq c_k\lambda$, $|Z|\geq c_k\lambda$ and $x_w\geq c_k$ for every $w\in W$, where $Z=\{z\in Z_0:N_A(z)=W\}$ and \(c_k>0\) depends only on \(k\).
\end{enumerate}
\end{lemma}

\begin{proof}
By Lemma~\ref{exact-defect-identity}, we have $e(L),\ e(B),\ \mathcal E_B,\ \mathcal E_L\leq{\mathfrak{D}_k(G)+\binom q2\leq D_q}$. So $|Q|\leq2e(L)\leq2D_q$. For a vertex
\(z\in L\) with \(d_A(z)<2(k-1)\), it follows that
\(x_z\leq(2(k-1)-d_A(z))x_z\). Then
$$
\sum_{\substack{z\in L\\d_A(z)<2(k-1)}}x_z\leq\sum_{\substack{z\in L\\d_A(z)<2(k-1)}}(2(k-1)-d_A(z))x_z\leq\mathcal E_L\leq D_q.
$$
Since $x_u\leq1$ for every vertex $u\in V(G)$, we have
\begin{align*}
\lambda=\sum_{v\in A}x_v&=\sum_{v\in C}x_v+\sum_{v\in Q}x_v+\sum_{v\in L\setminus Q}x_v \\
&\leq|C|+|Q|+\sum_{z\in Z_0}x_z+\sum_{\substack{z\in L\setminus Q\\d_A(z)<2(k-1)}}x_z
\leq|C|+3D_q+\sum_{z\in Z_0}x_z.
\end{align*}
So $\sum_{z\in Z_0}x_z\ge\lambda-|C|-3D_q$.

By $D_q=O_k(1)$ and $|C|\le4(k-1)$, there exists a constant \(c_k'>0\), depending only on \(k\), such that $\sum_{z\in Z_0}x_z\ge \lambda-c_k'$. Since $m$ is sufficiently large, we have \(\lambda\ge 2c_k'\). Then $\sum_{z\in Z_0}x_z>0$, that is, $Z_0$ is not empty. For \(z\in Z_0\), by definition, $N_A(z)\subseteq C$. Since \(|C|\le4(k-1)\), the number of possible neighborhoods \(N_A(z)\) is
at most $\binom{4(k-1)}{2(k-1)}$. For each \(2(k-1)\)-subset \(S\subseteq C\), define $Z(S)=\{z\in Z_0: N_A(z)=S\}$.
All nonempty sets \(Z(S)\) form a partition of \(Z_0\).
Consequently, by the weighted pigeonhole principle, there exists a \(2(k-1)\)-subset \(W\subseteq C\) such that
$$
\sum_{z\in Z(W)}x_z
\ge \frac{\lambda-c_k'}{\binom{4(k-1)}{2(k-1)}}\ge \frac{\lambda}{2\binom{4(k-1)}{2(k-1)}}.
$$
Let $Z=Z(W)$ and $c_k=\frac{1}{2\binom{4(k-1)}{2(k-1)}}>0$. Then $\sum_{z\in Z}x_z\ge c_k\lambda$.
Since \(x_z\leq1\) for every $z\in Z$, we have \(|Z|\geq\sum_{z\in Z}x_z\geq c_k\lambda\). For each \(w\in W\), we derive
$
\lambda x_w=\sum_{v\in N_G(w)}x_v\geq\sum_{z\in Z}x_z\geq c_k\lambda.
$
Thus $x_w\geq c_k$. We complete the proof.
\end{proof}

\begin{lemma}\label{Bbound}
It holds that \(|B|=O_k(1)\). 
\end{lemma}
\begin{proof}
Choose \(W,Z,c_k\) as in Lemma~\ref{weighted-neighborhood-class}.
Let \(B_1\) be the set of endpoints of the edges of \(G[B]\), and
put \(B_0=B\setminus B_1\). By Lemma~\ref{weighted-neighborhood-class},
\(|B_1|\leq2D_q\). For \(w\in B_0\), define
$$
X_w=\sum_{u\in A\setminus N_A(w)}x_u.
$$
The vertex \(w\) has no neighbor in \(B\), so by $\lambda x_w=\sum_{u\in N_A(w)}x_u$ and $\lambda=\lambda x_{u^\ast}=\sum_{u\in A}x_u$,
we have $X_w=\lambda(1-x_w)$.
Let $B_- =\{w\in B_0:X_w<\lambda/2\}$ and $B_+ =B_0\setminus B_-$. For \(w\in B_-\), we have \(x_w>1/2\). Then
\(d_A(w)\geq\sum_{u\in N_A(w)}x_u=\lambda x_w>\lambda/2\). Moreover,
\(x_w\leq1/2\) for \(w\in B_+\). Hence, by Lemmas \ref{exact-defect-identity} and \ref{weighted-neighborhood-class}, we have
\begin{equation}\label{eq:outer-weighted-bounds}
 \begin{aligned}
 \sum_{w\in B_-}X_w
 &\leq2\sum_{w\in B_-}d_A(w)(1-x_w)\leq2\mathcal E_B\leq2D_q,\\
 \sum_{w\in B_+}d_A(w)
 &\leq2\sum_{w\in B_+}d_A(w)(1-x_w)\leq2\mathcal E_B\leq2D_q.
 \end{aligned}
\end{equation}
Since \(G\) has no
isolated vertices and vertices of \(B_0\) have no neighbors in $B$,
it follows that \(d_A(w)\ge1\) for every $w\in B_0$. Therefore \(|B_+|\leq2D_q\).

Let $B'=\{w\in B_-:W\nsubseteq N_A(w)\}$, and $B''=B_-\setminus B'$. For every \(w\in B'\), there is a vertex \(u\in W\) such that $u\notin N_A(w)$, so
\(X_w\geq x_u\geq c_k\). By inequality \eqref{eq:outer-weighted-bounds}, we have
$$
c_k|B'|
\le \sum_{w\in B'}X_w
\le \sum_{w\in B_-}X_w
\le 2D_q.
$$
Hence \(|B'|\leq2D_q/c_k=O_k(1)\).

It remains to bound \(B''\). For every $w\in B''$, it follows that $W\subseteq N_A(w)$.
Put \(s=\nu(G[W])\leq k-1\), and consider the bipartite graph
\(\mathcal H=G[B'',Z]\).

Suppose first that \(s\leq k-2\). We claim that $\nu(\mathcal H)\leq k-s-2$. Otherwise, take a matching \(\{b_iz_i:b_i\in B'', z_i\in Z,1\leq i\leq k-1-s\}\) in
\(\mathcal H\), and a maximum matching $\{a_ja_j': 1\leq j\leq s\}$ of size \(s\) in \(G[W]\). Since $|W|=2(k-1)$, we have $|W\setminus\{a_1,a'_1,\ldots,a_s,a'_s\}|=2(k-1-s)$. Choose distinct vertices
\(w_1,\ldots,w_{k-1-s}\) from $W\setminus\{a_1,a'_1,\ldots,a_s,a'_s\}$. Recall that $W=N_A(z)$ for every $z\in Z$ and $W\subseteq N_A(w)$ for every $w\in B''$. Then $b_iz_iw_ib_i$ is a $K_3$ for $1\le i\le k-1-s$. Since $|Z|\geq c_k\lambda$ and $m$ is sufficiently large, we can choose \(s+1\) vertices $h_1,\ldots,h_s,h_{s+1}$ of \(Z\) different from $z_i$, \(1\leq i\leq k-1-s\). Then $h_pa_pa'_ph_p$ is a $K_3$ for $1\leq p\leq s$. Since $|W\setminus\{a_1,a'_1,\ldots,a_s,a'_s,w_1,\ldots,w_{k-1-s}\}|=k-1-s\geq1$, choose a vertex $w_0\in W\setminus\{a_1,a'_1,\ldots,a_s,a'_s,w_1,\ldots,w_{k-1-s}\}$. Then $w_0u^\ast h_{s+1}w_0$ is a $K_3$. Hence, $\{b_iz_iw_ib_i:1\le i\le k-1-s\}$, $\{h_pa_pa'_ph_p:1\leq p\leq s\}$ and $w_0u^\ast h_{s+1}w_0$ are \(k\) pairwise vertex-disjoint
triangles of $G$, a contradiction. Hence, $\nu(\mathcal H)\leq k-s-2$. By Theorem \ref{Konig}, \(\mathcal H\) has a vertex cover
\(S_B\sqcup S_Z\) with \(S_B\subseteq B''\), \(S_Z\subseteq Z\), and
\(|S_B|+|S_Z|=\nu(\mathcal H)\leq k-s-2\). If there is a vertex \(w\in B''\setminus S_B\), then
\(N_Z(w)\subseteq S_Z\), and consequently
\begin{equation*}\label{eq:outer-cover-contradiction}
 X_w\geq\sum_{z\in Z\setminus S_Z}x_z
      \geq c_k\lambda-|S_Z|>2D_q
\end{equation*}
for sufficiently large \(m\). This contradicts
\eqref{eq:outer-weighted-bounds}. Hence \(B''\subseteq S_B\) and so
\(|B''|\leq k-2-s\).

Suppose now that \(s=k-1\), so \(G[W]\) has a perfect matching $\{w_iw'_i:1\leq i\leq k-1\}$.
There is no edge between \(B''\) and \(Z\) in $G$. Otherwise, choose an edge $wz$ where $w\in B''$ and $z\in Z$. By Lemma \ref{weighted-neighborhood-class} (iii), we have $|Z|\geq c_k\lambda$. Since $m$ is sufficiently large, there are $k-1$ vertices $z_i\in Z\setminus\{z\}$ with $i\in [k-1]$. Note that $W=N_A(z')$ for every $z'\in Z$ and $W\subseteq N_A(w')$ for every $w'\in B''$. It follows that $wzw_1w$, $u^\ast w_1'z_1u^\ast$ and $\{w_iw'_iz_iw_i:2\leq i\leq k-1\}$ are \(k\) pairwise vertex-disjoint triangles of $G$, a contradiction. 
Thus, if $B''\neq\emptyset$, then we derive $Z\subseteq A\setminus N_A(w)$ for every \(w\in B''\). Hence \(X_w\geq \sum_{z\in Z}x_z\geq c_k\lambda>2D_q\), again contradicting
\eqref{eq:outer-weighted-bounds}. Therefore \(B''=\emptyset\).

Note that $B=B_1\sqcup B_+\sqcup B'\sqcup B''$. In both cases, we have \(|B|=O_k(1)\).
\end{proof}

\begin{lemma}\label{finite-kernel-reduction}
Let $V(G)=K\sqcup \mathcal I$ be a partition of $V(G)$ such that $\mathcal I$ is an independent set. Let \(N_1,\ldots,N_s\subseteq K\) be the distinct neighborhoods of the vertices of \(\mathcal I\) satisfying $\{N_1,\cdots, N_s\}=\{N_G(z):z\in \mathcal I\}$. For $1\leq i\leq s$, define $\mathcal I_i=\{z\in \mathcal I:N_K(z)=N_i\}$, $n_i=|\mathcal I_i|$, and $\mathbf{\chi}_i\in\{0,1\}^{K}$ the characteristic vector of \(N_i\). Let \(\mathbf{v}=\mathbf{x}|_K/\|\mathbf{x}|_K\|_2\), and set $M=\sum_{i=1}^s n_i\mathbf{\chi}_i\mathbf{\chi}_i^{\mathsf T}$.
Then
\begin{equation}\label{eq:reduced-eigenvalue}
 \lambda^2\mathbf{v}=\lambda A(G[K])\mathbf{v}+M\mathbf{v}.
\end{equation}
Define $\zeta=\operatorname{tr}M-\mathbf{v}^{\mathsf T}M\mathbf{v}$. Then
we have $\zeta\geq0$ and
\begin{equation}\label{eq:kernel-scalar-identity}
 \lambda^2=m-e(G[K])-\zeta+\lambda \mathbf{v}^{\mathsf T}A(G[K])\mathbf{v}.
\end{equation}
If $|K|=O_k(1)$, then $0\leq\zeta=O_k(\lambda)$.
\end{lemma}

\begin{proof}
Since $N_G(u)=N_i$ for any vertex $u\in \mathcal I_i$ ($1\leq i\leq s$), the vector $\mathbf{x}$ has the same coordinates corresponding to the vertices of $\mathcal I_i$. Let $\eta_i=\frac{x_u}{\|\mathbf{x}|_K\|_2}$ for $u\in\mathcal I_i$. So $\lambda\eta_i=\chi_i^{\mathsf T}\mathbf{v}$. By $A(G)\mathbf{x}=\lambda\mathbf{x}$, we have
\begin{align*}
\lambda \mathbf{v}&
=A(G[K])\mathbf{v}+\sum_i n_i\frac{\chi_i^{\mathsf T}\mathbf{v}}{\lambda}\mathbf{\chi}_i
=A(G[K])\mathbf{v}+\frac{1}{\lambda}\sum_{i=1}^{s}n_i\chi_i\chi_i^\mathsf{T}\mathbf{v}
 =A(G[K])\mathbf{v}+\frac{1}{\lambda}M\mathbf{v}.
\end{align*}
Hence equation (\ref{eq:reduced-eigenvalue}) holds.

For \(x,y\in K\), $M_{xy}
 =\sum_{i=1}^{s}n_i(\chi_i)_{x}(\chi_i)_{y}
 =|N_{\mathcal I}(x)\cap N_{\mathcal I}(y)|$. Thus, when $x\neq y$, the entry $M_{xy}$ counts the common neighbors in ${\mathcal I}$ of $x$
and $y$, and $M_{xx}=d_{\mathcal I}(x)$.
Equivalently, \(M=\mathcal R\mathcal R^{\mathsf T}\), where \(\mathcal R\) is the bipartite adjacency matrix between \(K\) and \({\mathcal I}\), that is, $\mathcal R_{ii'}=1$ if the vertex $v_i\in K$ is adjacent to the vertex $u_{i'}\in {\mathcal I}$; and $\mathcal R_{ii'}=0$ otherwise.
Thus \(\operatorname{tr}M=e(K,{\mathcal I})=m-e(G[K])\).
Left-multiplying both sides by $\mathbf{v}^\mathsf{T}$ for equation (\ref{eq:reduced-eigenvalue}) and by $\|\mathbf{v}\|_2=1$, we obtain
\[
 \lambda^2=\lambda \mathbf{v}^{\mathsf T}A(G[K])\mathbf{v}+\mathbf{v}^{\mathsf T}M\mathbf{v}.
\]
Since $\zeta=\operatorname{tr}M-\mathbf{v}^{\mathsf T}M\mathbf{v}$, it follows that
$$
\lambda^2
 =\operatorname{tr}M-\zeta+\lambda \mathbf{v}^{\mathsf T}A(G[K])\mathbf{v}=m-e(G[K])-\zeta+\lambda \mathbf{v}^{\mathsf T}A(G[K])\mathbf{v}.
$$
Thus equation \eqref{eq:kernel-scalar-identity} holds. By \(\operatorname{tr}M=e(K,{\mathcal I})\) and $M=\sum_{i=1}^s n_i\mathbf{\chi}_i\mathbf{\chi}_i^{\mathsf T}$,
we have
\begin{align}\label{eq:4.010}
 \zeta=\sum_{i=1}^s n_i\bigl(|N_i|-(\mathbf{\chi}_i^{\mathsf T}\mathbf{v})^2\bigr).
\end{align}
Because
\((\mathbf{\chi}_i^{\mathsf T}\mathbf{v})^2\leq\|\mathbf{\chi}_i\|_2^2\|\mathbf{\mathbf{v}}\|_2^2=|N_i|\) for $1\leq i\leq s$, we derive $\zeta\geq0$.

Finally, if $|K|=O_k(1)$, since $\mathbf{v}^{\mathsf T}A(G[K])\mathbf{v}\leq\lambda(G[K])\leq |K|-1=O_k(1)$ and $\mathfrak D_k(G)=m-(2k-1)(k-1)+2(k-1)\lambda-\lambda^2\leq\delta_\ell$, by equation \eqref{eq:kernel-scalar-identity}, we get
\begin{align}\label{eq:4.011}
 0\leq\zeta
 =m-e(G[K])+\lambda \mathbf{v}^{\mathsf T}A(G[K])\mathbf{v}-\lambda^2
 = O_k(\lambda),
\end{align}
as required.
\end{proof}

Note that $L\setminus Q$ is independent. In the following, we may set $K=\{u^\ast\}\sqcup C\sqcup B\sqcup Q$ and $\mathcal I=L\setminus Q$. By Lemmas~\ref{exact-defect-identity}, \ref{weighted-neighborhood-class} and~\ref{Bbound}, we have \(|K|=O_k(1)\). 
\begin{proposition}
\label{linear-neighborhood-type}
$
|\mathcal I|=\Omega_k(m)$ and $\bigl|\{z\in \mathcal I:d_A(z)<2(k-1)\}\bigr|
\leq D_q\lambda=O_k(\lambda)$. For every $1\leq i\leq s$ with
\(n_i=\Omega_k(m)\), denote by $\mathbf{w}_i=\frac{\mathbf{\chi}_i}{\sqrt{|N_i|}}$. Then $\|\mathbf{v}-\mathbf{w}_i\|_2=O_k(\lambda^{-1/2})$.
Moreover, there is a $j\in [s]$ with \(n_j=\Omega_k(m)\) satisfying
\begin{equation}\label{eq:linear-type-form}
N_j=\{u^\ast\}\sqcup W_j,
\qquad
W_j\subseteq C,\quad |W_j|=2(k-1),\quad G[N_j]\cong K_{2k-1}.
\end{equation}
\end{proposition}
\begin{proof}
Since \(\mathcal I\) is independent and every neighbor of a vertex
of \(\mathcal I\) belongs to \(K\), we have
\(
m=e(G[K])+\sum_{z\in \mathcal I}d_K(z)
\leq e(G[K])+|K|\,|\mathcal I|.
\)
By \(|K|=O_k(1)\) and
\(e(G[K])\leq\binom{|K|}{2}=O_k(1)\), we obtain $|\mathcal I|=\Omega_k(m)$.

Since every vertex \(z\in \mathcal I\subseteq A=N_G(u^\ast)\) is adjacent to
\(u^\ast\), we have
\(
\lambda x_z
=\sum_{u\in N_G(z)}x_u
\geq x_{u^\ast}=1,
\)
and therefore \(x_z\geq\lambda^{-1}\). Put $S=\{z\in \mathcal I:d_A(z)<2(k-1)\}$. By Lemma \ref{weighted-neighborhood-class} (ii),
\[
\frac{|S|}{\lambda}
\leq \sum_{z\in S}x_z
\leq D_q.
\]
Thus $|S|\leq D_q\lambda=O_k(\lambda)$.

Note that \(|K|=O_k(1)\). For any \(1\leq i\leq s\) with \(n_i=\Omega_k(m)\), by Lemma~\ref{finite-kernel-reduction} and \eqref{eq:4.010}, we have $O_k(\lambda)=\zeta\geq n_i\bigl(|N_i|-(\chi_i^{\mathsf T}\mathbf{v})^2\bigr)$. Since \(\lambda^2\leq\operatorname{tr}(A^2(G))=2m\), we have
\(n_i=\Omega_k(m)=\Omega_k(\lambda^2)\), and hence
\(
|N_i|-(\chi_i^{\mathsf T}\mathbf{v})^2=O_k(\lambda^{-1}).
\)
Since \(\mathbf{v}\) is a positive unit vector and \(\mathbf{w}_i\) is a nonnegative unit vector, we deduce $0\leq\mathbf{w}_i^{\mathsf T}\mathbf{v}\leq\|\mathbf{w}_i\|_2\|\mathbf{v}\|_2=1$. Thus
\[
\begin{aligned}
\|\mathbf{v}-\mathbf{w}_i\|_2^2
   =\|\mathbf{v}\|_2^2+\|\mathbf{w}_i\|_2^2-2\mathbf{w}_i^{\mathsf T}\mathbf{v}
   \leq2(1-(\mathbf{w}_i^{\mathsf T}\mathbf{v})^2)
   =\frac{2}{|N_i|}
  \bigl(|N_i|-(\chi_i^{\mathsf T}\mathbf{v})^2\bigr)
=O_k(\lambda^{-1}).
\end{aligned}
\]

By $\lambda^2\leq2m$, we derive $|S|=O_k(\sqrt m)=o(m)$.
Consequently, $|{\mathcal I}\setminus S|=\Omega_k(m)$. Note that $d_A(u)\leq2(k-1)$ for any vertex $u\in {\mathcal I}\subseteq L$. So $d_A(u)=2(k-1)$ for any vertex $u\in {\mathcal I}\setminus S$. 
Since $K$ has at most \(2^{|K|}=O_k(1)\) subsets and $N_G({\mathcal I}\setminus S)\subseteq K$, by the pigeonhole principle, there is a subset \(N\subseteq K\) such that there is at least
$\frac{|{\mathcal I}\setminus S|}{2^{|K|}}=\Omega_k(m)$ vertices in ${\mathcal I}\setminus S$ whose neighborhood is $N.$ As $\{z\in {\mathcal I}\setminus S:N_K(z)=N\}\subseteq\{z\in {\mathcal I}:N_K(z)=N\}$, we have $N=N_j$ for some $j\in[s]$, and so $n_j=\Omega_k(m)$. Choose \(z\in {\mathcal I}\setminus S\) with \(N_G(z)=N_j\).
Since \(z\in {\mathcal I}=L\setminus Q\), it has no neighbor in \(L\).
Thus all its \(2(k-1)\) neighbors in \(A\) belong to \(C\).
Also, \(z\) is adjacent to \(u^\ast\), and every remaining
neighbor of \(z\) belongs to \(B\).
Set $W_j=N_j\cap C$. 
We obtain $N_j=\{u^\ast\}\sqcup W_j\sqcup (N_j\cap B)$ and $|W_j|=2(k-1)$. Then \(|N_j|\geq2k-1\).

We observe that \(\nu(G[N_j])\leq k-1\).
Otherwise, choose a matching $\{p_ip'_i: 1\leq i\leq k\}$ in \(G[N_j]\). Since $n_j=\Omega_k(m)$, there are \(k\) distinct vertices $q_j$ with $j\in [k]$ in \(\mathcal I_j\setminus S\). Then $p_ip'_iq_ip_i$ where $i\in[k]$ are \(k\) pairwise vertex-disjoint triangles of $G$.
This contradicts that \(G\) is \(kK_3\)-free.

By inequality \eqref{boundofnoperfectmatching}, we obtain
\begin{equation}\label{eq:dominant-density}
\frac{e(G[N_j])}{|N_j|}\leq k-1,
\end{equation}
with equality only if \(|N_j|=2k-1\) and
\(G[N_j]\cong K_{2k-1}\). Suppose that inequality \eqref{eq:dominant-density} is strict. Then \((k-1)|N_j|-e(G[N_j])\geq1\). It follows that
\[
\frac{e(G[N_j])}{|N_j|}
\leq k-1-\frac{1}{|N_j|}
\leq k-1-|K|^{-1}.
\]
By $\lambda(G[K])\leq|K|-1=O_k(1)$, we deduce
\begin{align*}
\mathbf{v}^{\mathsf T}A(G[K])\mathbf{v}
-\mathbf{w}_i^{\mathsf T}A(G[K])\mathbf{w}_i&\leq\left|\mathbf{v}^{\mathsf T}A(G[K])\mathbf{v}
-\mathbf{w}_i^{\mathsf T}A(G[K])\mathbf{w}_i\right|\\
&=\left|(\mathbf{v}-\mathbf{w}_i)^{\mathsf T}A(G[K])(\mathbf{v}+\mathbf{w}_i)\right|\\
&\leq\|\mathbf{v}-\mathbf{w}_i\|_2\|A(G[K])\|_2\|\mathbf{v}+\mathbf{w}_i\|_2\\
&\leq\|\mathbf{v}-\mathbf{w}_i\|_2\|A(G[K])\|_2(\|\mathbf{v}\|_2+\|\mathbf{w}_i\|_2)\\
&=2\lambda(G[K])\|\mathbf{v}-\mathbf{w}_i\|_2\\
&=O_k(\lambda^{-1/2}).
\end{align*}
Since
\[
\mathbf{w}_i^{\mathsf T}A(G[K])\mathbf{w}_i=\frac{1}{|N_i|}\mathbf{\chi}_i^{\mathsf T}A(G[K])\mathbf{\chi}_i
 =\frac{1}{|N_i|}\sum_{x,y\in N_i}(A(G[K]))_{xy}
 =\frac{2e(G[N_i])}{|N_i|},
\]
it follows that
\(
\mathbf{v}^{\mathsf T}A(G[K])\mathbf{v}
\leq2(k-1-|K|^{-1})+O_k(\lambda^{-1/2}).
\)
Substituting it into equation \eqref{eq:kernel-scalar-identity}, and by
\(e(G[K]),\zeta\geq0\), we have
\(
\lambda^2
\leq m+2(k-1-|K|^{-1})\lambda
+O_k(\lambda^{1/2}).
\)
Note that $\mathfrak D_k(G)=m-\gamma+2(k-1)\lambda-\lambda^2\leq\delta_\ell$, this implies
\[
2|K|^{-1}\lambda
\leq O_k(\lambda^{1/2})+\delta_\ell+\gamma=O_k(\lambda^{1/2}),
\]
which is impossible for sufficiently large $m$.
Thus equality holds in inequality \eqref{eq:dominant-density}, as desired.
\end{proof}

\begin{proposition}\label{C2r}
{It} holds that \(|C|=2(k-1)\).
\end{proposition}
\begin{proof}
By Proposition~\ref{linear-neighborhood-type}, there is a \(j\in [s]\) with \(n_j=\Omega_k(m)\) satisfying \(N_j=\{u^\ast\}\sqcup W_j\), \(|W_j|=2(k-1)\) and
\(G[W_j]\cong K_{2(k-1)}\). If \(|C|>2(k-1)\), choose
\(c_0\in C\setminus W_j\). Since \(d_A(c_0)>2(k-1)=|W_j|\), there is a vertex
\(z_0\in N_A(c_0)\setminus W_j\).
Let \(\{a_1b_1,\ldots,a_{k-1}b_{k-1}\}\) be a perfect matching of \(G[W_j]\).
As \(n_j=\Omega_k(m)\), choose distinct
\(z_1,\ldots,z_{k-1}\in\mathcal I_j\setminus\{z_0\}\).
Then $\{z_ia_ib_iz_i:1\leq i\leq k-1\}$ and $u^\ast c_0 z_0u^\ast$ are \(k\) pairwise vertex-disjoint triangles of $G$, a contradiction. Thus \(|C|\leq2(k-1)\). By Lemma~\ref{exact-defect-identity}, we have $|C|=2(k-1)$.
\end{proof}
\section{The proof of Theorem~\ref{main-theorem}}\label{s5}
In this section, based on the foregoing results, we further characterize the structure of extremal graphs in detail, thus completing the proof of Theorem~\ref{main-theorem}. For convenience, we retain the notation and terminology introduced in earlier sections.

By Proposition~\ref{linear-neighborhood-type}, it follows that there is a $j\in[s]$ with
\(N_j=\{u^\ast\}\sqcup W_j\), $G[N_j]\cong K_{2k-1}$ and \(n_j=\Omega_k(m)\).
Since \(|W_j|=2(k-1)\) and \(W_j\subseteq C\), by Proposition~\ref{C2r}, we have \(W_j=C\). 

\begin{lemma}\label{bounded-multiplicity-of-nonprincipal-types}
$|{\mathcal I}\setminus \mathcal I_j|=O_k(1)$.
\end{lemma}
\begin{proof}
Denote by $\mathbf{w}:=\mathbf{w}_j,$ by Proposition~\ref{linear-neighborhood-type}, we have $\|\mathbf{v}-\mathbf{w}\|_2=O_k(\lambda^{-1/2})$. For \(i\in [s]\setminus \{j\}\), define $g_i(\mathbf{w})=|N_i|-(\chi_i^{\mathsf T}\mathbf{w})^2$.
By the Cauchy--Schwarz inequality,
\[
 (\mathbf{\chi}_i^{\mathsf T}\mathbf{w})^2
 \leq\|\mathbf{\chi}_i\|_2^2\|\mathbf{w}\|_2^2
 =|N_i|,
\]
and hence \(g_i(\mathbf{w})\geq0\). Equality holds only if
\(\mathbf{\chi}_i\) is a scalar multiple of \(\mathbf{\chi}_j\). Since \({\mathcal I}\subseteq A=N(u^\ast)\), \(u^\ast\in N_i\), and so \(N_i\neq\emptyset\). That is, $\mathbf{\chi}_i\neq\mathbf{0}$. Since \(\mathbf{\chi}_i\) and \(\mathbf{\chi}_j\) are nonzero \(0\)--\(1\) vectors, equality is possible only when $N_i=N_j,$ i.e.,
\(i=j\). It follows that $g_i(\mathbf{w})>0$ whenever $i\neq j$.
More explicitly,
\[
 \mathbf{\chi}_i^{\mathsf T}\mathbf{w}
 =\frac{|N_i\cap N_j|}{\sqrt {2k-1}},
 \qquad
 g_i(\mathbf{w})
 =\frac{(2k-1)|N_i|-|N_i\cap N_j|^2}{2k-1}.
\]
We may therefore fix a constant
\(\eta_k>0\), for instance \(\eta_k=1/(2k-1)\), such that $g_i(\mathbf{w})\geq\eta_k$ for every $i\neq j$.
Since \(|N_i|\leq|K|=O_k(1)\), we have
\[ |g_i(\mathbf{v})-g_i(\mathbf{w})|
 = \left|(\mathbf{\chi}_i^{\mathsf T}\mathbf{w})^2-(\mathbf{\chi}_i^{\mathsf T}\mathbf{v})^2\right|
 \leq
 \|\mathbf{\chi}_i\|_2^2\|\mathbf{w}-\mathbf{v}\|_2\|\mathbf{w}+\mathbf{v}\|_2
 \leq2|K|\,\|\mathbf{w}-\mathbf{v}\|_2
 =O_k(\lambda^{-1/2}).
 \]
Hence, $|N_i|-(\mathbf{\chi}_i^{\mathsf T}\mathbf{v})^2=g_i(\mathbf{v})\geq\frac{\eta_k}{2}$ for any $i\neq j$.
By \eqref{eq:4.010} and \eqref{eq:4.011},
$$
O_k(\lambda)=\zeta=\sum_{i=1}^s n_i\bigl(|N_i|-(\mathbf{\chi}_i^{\mathsf T}\mathbf{v})^2\bigr)\geq\sum_{i\neq j}n_ig_i(\mathbf{v})\geq\frac{\eta_k}{2}\left(\sum_{i\neq j}n_i\right).
$$
Thus $\sum_{i\neq j}n_i=O_k(\lambda)$.

Let $M'=\sum_{i\neq j}n_i\mathbf{\chi}_i\mathbf{\chi}_i^{\mathsf T}$.
Since for each \(i\), $\|\mathbf{\chi}_i\mathbf{\chi}_i^{\mathsf T}\|_2=\|\mathbf{\chi}_i\|_2^2=|N_i|$, one has
\[
 \begin{aligned}
 \|M'\|_2
 &\leq
 \sum_{i\neq j}
 n_i\|\mathbf{\chi}_i\mathbf{\chi}_i^{\mathsf T}\|_2
 =\sum_{i\neq j}n_i|N_i|
 \leq |K|\left(\sum_{i\neq j}n_i\right)
 =O_k(\lambda).
 \end{aligned}
\]
Since $M=n_j\mathbf{\chi}_j\mathbf{\chi}_j^{\mathsf T}+M'$, by equation \eqref{eq:reduced-eigenvalue}, we have
\begin{align}\label{vequation}
 \lambda^2\mathbf{v}
 =\lambda A(G[K])\mathbf{v}
 +n_j\mathbf{\chi}_j\mathbf{\chi}_j^{\mathsf T}\mathbf{v}
 +M'\mathbf{v}.
\end{align}

Since $\|\mathbf{w}\|_2=1$, we have $\|I_{|K|}-\mathbf{w}\mathbf{w}^{\mathsf T}\|_2=1$. By \(\mathbf{\chi}_j=\sqrt {2k-1}\mathbf{w}\), we have $(I_{|K|}-\mathbf{w}\mathbf{w}^{\mathsf T})\mathbf{\chi}_j\mathbf{\chi}_j^{\mathsf T}\mathbf{v}=0$.
Left-multiplying both sides by \(I_{|K|}-\mathbf{w}\mathbf{w}^{\mathsf T}\) for equation (\ref{vequation}) gives
\begin{equation}
 \lambda^2(I_{|K|}-\mathbf{w}\mathbf{w}^{\mathsf T})\mathbf{v}
 =\lambda (I_{|K|}-\mathbf{w}\mathbf{w}^{\mathsf T})A(G[K])\mathbf{v}
 +(I_{|K|}-\mathbf{w}\mathbf{w}^{\mathsf T})M'\mathbf{v}.
 \label{eq:original-7-12}
\end{equation}
Since $\|A(G[K])\|_2=\lambda(G[K])\leq|K|-1=O_k(1)$ and $\|\mathbf{v}\|_2=1$,
we deduce
\[
 \begin{aligned}
 \|\lambda (I_{|K|}-\mathbf{w}\mathbf{w}^{\mathsf T})A(G[K])\mathbf{v}\|_2
 &\leq\lambda\|I_{|K|}-\mathbf{w}\mathbf{w}^{\mathsf T}\|_2\|A(G[K])\|_2\|\mathbf{v}\|_2
 =O_k(\lambda),\\
 \|(I_{|K|}-\mathbf{w}\mathbf{w}^{\mathsf T})M'\mathbf{v}\|_2
 &\leq\|I_{|K|}-\mathbf{w}\mathbf{w}^{\mathsf T}\|_2\|M'\|_2\|\mathbf{v}\|_2
 =O_k(\lambda).
 \end{aligned}
\]
It follows from equation \eqref{eq:original-7-12} that $\|(I_{|K|}-\mathbf{w}\mathbf{w}^{\mathsf T})\mathbf{v}\|_2=O_k(\lambda^{-1})$.

Write $\mathbf{v}=\alpha \mathbf{w}+\mathbf{z},$ where $\alpha=\mathbf{w}^{\mathsf T}\mathbf{v}$. Then $\mathbf{z}=(I_{|K|}-\mathbf{w}\mathbf{w}^{\mathsf T})\mathbf{v},$ which is orthogonal to $\mathbf{w},$ and $\|\mathbf{v}-\mathbf{w}\|_2^2=\|\mathbf{v}\|_2^2+\|\mathbf{w}\|_2^2
-2\mathbf{w}^\mathsf{T}\mathbf{v}=2(1-\alpha)$.
Note that $\|\mathbf{v}-\mathbf{w}\|_2=O_k(\lambda^{-1/2})$. It follows that
\(
 \alpha
 =1-\frac12\|\mathbf{v}-\mathbf{w}\|_2^2
 =1-O_k(\lambda^{-1}),
\)
and in particular \(\alpha\to1\).  On the other hand, the unit norm of
\(\mathbf{v}\) gives
\[
 1=\|\mathbf{v}\|_2^2=\|\alpha \mathbf{w}+\mathbf{z}\|_2^2=(\alpha \mathbf{w}+\mathbf{z})^\mathsf{T}(\alpha \mathbf{w}+\mathbf{z})=\alpha^2+\|\mathbf{z}\|_2^2.
\]
Since \(\|\mathbf{z}\|_2=O_k(\lambda^{-1})\), we obtain
\[
 1-\alpha
 =\frac{1-\alpha^2}{1+\alpha}
 =\frac{\|\mathbf{z}\|_2^2}{1+\alpha}
 =O_k(\lambda^{-2}).
\]
Therefore
\[
 \begin{aligned}
 \|\mathbf{v}-\mathbf{w}\|_2^2
 &=\|(\alpha-1)\mathbf{w}+\mathbf{z}\|_2^2=(\alpha-1)^2+\|\mathbf{z}\|_2^2
 =O_k(\lambda^{-2}),
 \end{aligned}
\]
and hence $ \|\mathbf{v}-\mathbf{w}\|_2=O_k(\lambda^{-1})$.

Since \(A(G[K])\) is symmetric,
\[
 \mathbf{v}^{\mathsf T}A(G[K])\mathbf{v}-\mathbf{w}^{\mathsf T}A(G[K])\mathbf{w}
 =(\mathbf{v}-\mathbf{w})^{\mathsf T}A(G[K])(\mathbf{v}+\mathbf{w}).
\]
Thus,
\[
 \begin{aligned}
 \left|\mathbf{v}^{\mathsf T}A(G[K])\mathbf{v}-\mathbf{w}^{\mathsf T}A(G[K])\mathbf{w}\right|
 &\leq
 \|\mathbf{v}-\mathbf{w}\|_2\|A(G[K])\|_2\|\mathbf{v}+\mathbf{w}\|_2\\
 &\leq2\|A(G[K])\|_2\|\mathbf{v}-\mathbf{w}\|_2
 =O_k(\lambda^{-1}).
 \end{aligned}
\]
Since \(G[N_j]\cong K_{2k-1}\), we have
\[
 \begin{aligned}
 \mathbf{w}^{\mathsf T}A(G[K])\mathbf{w}
 &=\frac{1}{2k-1}\mathbf{\chi}_j^{\mathsf T}A(G[K])\mathbf{\chi}_j
 =\frac{2e(G[N_j])}{2k-1}
 =\frac{2\binom{2k-1}{2}}{2k-1}
 =2(k-1).
 \end{aligned}
\]
We conclude that
\begin{equation*}
 \mathbf{v}^{\mathsf T}A(G[K])\mathbf{v}
 =\mathbf{w}^{\mathsf T}A(G[K])\mathbf{w}+O_k(\lambda^{-1})
 =2(k-1)+O_k(\lambda^{-1}).
\end{equation*}
Substituting the above equation into equation \eqref{eq:kernel-scalar-identity}, we obtain
\[
 \begin{aligned}
 \lambda^2
 &=m-e(G[K])-\zeta
   +\lambda\bigl(2(k-1)+O_k(\lambda^{-1})\bigr)
 =m-e(G[K])-\zeta+2(k-1)\lambda+O_k(1).
 \end{aligned}
\]
Equivalently, $\zeta=m+2(k-1)\lambda-\lambda^2-e(G[K])+O_k(1)$.
By $\mathfrak D_k(G)=m-(2k-1)(k-1)+2(k-1)\lambda-\lambda^2\leq\delta_\ell$ and \(e(G[K])\geq0\), we have $\zeta= O_k(1)$. Recall that $\zeta\geq\frac{\eta_k}{2}\sum_{i\neq j}n_i$. It follows that $ \sum_{i\neq j}n_i=O_k(1)$, as required.
\end{proof}


Now, there is a partition $V(G)=N_j\sqcup \mathcal I_j\sqcup X$ satisfying $|X|=O_k(1)$. Note that $N(u)=N_j$ for any $u\in \mathcal I_j$, so \(E(\mathcal I_j,X)=\emptyset\). 
Let \(Y\) be the bipartite adjacency matrix between \(N_j\) and \(X\), 
and let \(\mathbf{\xi}=\mathbf{x}|_{N_j}/\|\mathbf{x}|_{N_j}\|_2\) and \(\mathbf{\sigma}=\mathbf{x}|_{X}/\|\mathbf{x}|_{N_j}\|_2\). 

Fix \(x\in X\), we have
\[
 \lambda\mathbf{\sigma}_x
 =\sum_{u\in N_{N_j}(x)}\mathbf{\xi}_u
  +\sum_{v\in N_{X}(x)}\mathbf{\sigma}_v
  =(Y^\mathsf{T}\mathbf{\xi})_x+(A(G[X])\mathbf{\sigma})_x.
\]
Hence $(\lambda I_{|X|}-A(G[X]))\mathbf{\sigma}=Y^\mathsf{T}\mathbf{\xi}$.
Since \(|X|=O_k(1)\),
\(\lambda(G[X])\leq|X|-1=O_k(1)\). As $m$ is sufficiently large, we have
\(\lambda>\lambda(G[X])\). Thus
\(\lambda I_{|X|}-A(G[X])\) is nonsingular. It follows that $\sigma=(\lambda I_{|X|}-A(G[X]))^{-1}Y^\mathsf{T}\mathbf{\xi}$. The coordinate of $\mathbf{x}/\|\mathbf{x}|_{N_j}\|_2$ corresponding to every vertex $v$ in \(\mathcal I_j\) is
$$
\frac{\sum_{u\in N_{N_j}(v)}\mathbf{\xi}_u}{\lambda}=\frac{\sum_{u\in N_j}\mathbf{\xi}_u}{\lambda}=\frac{\mathbf1^\mathsf{T}\mathbf{\xi}}{\lambda}.
$$
Therefore,
\(
 \lambda\mathbf{\xi}
 =(J_{2k-1}-I_{2k-1})\mathbf{\xi}+\frac{n_j}{\lambda}J_{2k-1}\mathbf{\xi}
 +Y\mathbf{\sigma}.
\)
Multiplying both sides by \(\lambda\), we obtain
\begin{align}\label{xiequation}
 \lambda^2\mathbf{\xi}
 =\lambda(J_{2k-1}-I_{2k-1})\mathbf{\xi}+n_jJ_{2k-1}\mathbf{\xi}
  +\lambda Y(\lambda I_{|X|}-A(G[X]))^{-1}Y^\mathsf{T}\mathbf{\xi}.
\end{align}

Note that $|X|=O_k(1)$ and \(Y\) is a \(0\)--\(1\) matrix, we obtain
\(
 \|Y\|_2\le \|Y\|_F
 \le \sqrt{(2k-1)|X|}=O_k(1).
\)
Moreover, since $\|A(G[X])\|_2=\lambda(A(G[X]))=O_k(1)$, and $m$ is sufficiently large, we have \(\lambda\ge 2\|A(G[X])\|_2\). By the Neumann-series estimate, one gets
 \begin{align*}
 \|(\lambda I_{|X|}-A(G[X]))^{-1}\|_2&=\frac{1}{\lambda}\|(I_{|X|}-\frac{A(G[X])}{\lambda})^{-1}\|_2\\
 &=\frac{1}{\lambda}\|\sum_{j\geq0}(\frac{A(G[X])}{\lambda})^j\|_2\\
 &\leq\frac{1}{\lambda}\sum_{j\geq0}\|(\frac{A(G[X])}{\lambda})^j\|_2\\
 &\leq\frac{1}{\lambda}\sum_{j\geq0}\frac{\|A(G[X])\|_2^j}{\lambda^j}\\
 &\leq\frac{1}{\lambda}\frac{1}{1-\|A(G[X])\|_2/\lambda}\\
 &=\frac{1}{\lambda-\|A(G[X])\|_2}
 \le \frac{2}{\lambda}.
 \end{align*}
Then
\begin{align}\label{Yequation}
 \bigl\|\lambda Y(\lambda I_{|X|}-A(G[X]))^{-1}Y^\mathsf{T}\bigr\|_2
 &\le \lambda\|Y\|_2^2
       \|(\lambda I_{|X|}-A(G[X]))^{-1}\|_2=O_k(1).
 \end{align}

Write $\xi=\frac{\alpha_0}{\sqrt{2k-1}} \mathbf 1_{2k-1}+\mathbf{z}_0$, where $\alpha_0=\frac{\mathbf 1_{2k-1}^{\mathsf T}\mathbf{\xi}}{\sqrt{2k-1}}$. Then $\mathbf{z}_0=(I_{2k-1}-\frac{J_{2k-1}}{2k-1})\xi,$ which is orthogonal to $\mathbf 1_{2k-1}$.
Since \(\|\xi\|_2=1\), we have
\(\alpha_0^2+\|\mathbf{z}_0\|_2^2=1\). As both $\xi$ and $\mathbf 1_{2k-1}$ are positive vectors, we have \(\alpha_0>0\), and hence
\[
 \alpha_0=\sqrt{1-\|\mathbf{z}_0\|_2^2}=1-\frac{\|\mathbf{z}_0\|_2^2}{2}
 +O(\|\mathbf{z}_0\|_2^4).
\]
On the other hand, Left-multiplying both sides by $I_{2k-1}-\frac{J_{2k-1}}{2k-1}$ for equation (\ref{xiequation}), we deduce
\((\lambda^2+\lambda)\mathbf{z}_0
=(I_{2k-1}-\frac{J_{2k-1}}{2k-1})\lambda Y(\lambda I_{|X|}-A(G[X]))^{-1}Y^\mathsf{T}\mathbf{\xi}.\) By inequality \eqref{Yequation} and \(\|I_{2k-1}-\frac{J_{2k-1}}{2k-1}\|_2=\|\xi\|_2=1\), one gets
\( \|\mathbf{z}_0\|_2 \le \frac{O_k(1)}{\lambda^2+\lambda} =O_k(\lambda^{-2}) =O_k(m^{-1}).
\) Then $\alpha_0=1-\frac{\|\mathbf{z}_0\|_2^2}{2}
 +O(\|\mathbf{z}_0\|_2^4)=1+O_k(m^{-2}).$
Consequently,
\begin{align}\notag
\|\xi-\frac{\mathbf 1_{2k-1}}{\sqrt{2k-1}} \|_2&=\|\frac{\alpha_0}{\sqrt{2k-1}} \mathbf 1_{2k-1}+\mathbf{z}_0-\frac{\mathbf 1_{2k-1}}{\sqrt{2k-1}}\|_2=\|\frac{\alpha_0-1}{\sqrt{2k-1}}\mathbf 1_{2k-1}+\mathbf{z}_0\|_2\\ \label{eq:5.01}
&\leq|\alpha_0-1|\|\frac{\mathbf 1_{2k-1}}{\sqrt{2k-1}}\|_2+\|\mathbf{z}_0\|_2=O_k(m^{-1}).
\end{align}

Since
\(J_{2k-1}\mathbf 1_{2k-1}=(2k-1)\mathbf 1_{2k-1}\) and \(J_{2k-1}\mathbf{z}_0=\mathbf{0}\), one has
 \begin{align}\notag
 \xi^\mathsf{T}(J_{2k-1}-I_{2k-1})\xi&=\xi^\mathsf{T}J_{2k-1}\xi
 -\xi^\mathsf{T}I_{2k-1}\xi=(2k-1)\alpha_0^2-1
 =(2k-1)(1-\|\mathbf{z}_0\|_2^2)-1\\ \label{eq:5.02}
 &=(2k-2)-(2k-1)\|\mathbf{z}_0\|_2^2
 =2k-2+O_k(m^{-2}),
 \end{align}

Similarly, using \(n_j=O_k(m)\), we obtain
 \begin{align}\label{eq:5.03}
 n_j\xi^\mathsf{T} J_{2k-1}\xi
 &=(2k-1)n_j\alpha_0^2
 =(2k-1)n_j-(2k-1)n_j\|\mathbf{z}_0\|_2^2
 =(2k-1)n_j+O_k(m^{-1}).
 \end{align}
For each \(x\in X\), we obtain
\[
 (Y^\mathsf{T} \mathbf 1_{2k-1})_x
 ={|N_{N_j}(x)|}.
\]
Since $|N_{N_j}(x)|\leq |N_j|=2k-1$, we have \(\|Y^\mathsf{T} \mathbf 1_{2k-1}\|_2=O_k(1)\).

Furthermore, by \eqref{eq:5.01},
\[
 \|Y^\mathsf{T}(\xi-\frac{\mathbf 1_{2k-1}}{\sqrt{2k-1}})\|_2
 \le \|Y\|_2\|\xi-\frac{\mathbf 1_{2k-1}}{\sqrt{2k-1}}\|_2
 =O_k(m^{-1}).
\]
It follows that
\[
 \begin{aligned}
 \xi^\mathsf{T} YY^\mathsf{T}\xi
 &=\|Y^\mathsf{T}\xi\|_2^2=\|Y^\mathsf{T} \frac{\mathbf 1_{2k-1}}{\sqrt{2k-1}}+Y^\mathsf{T} (\xi-\frac{\mathbf 1_{2k-1}}{\sqrt{2k-1}})\|_2^2\\
 &=\|Y^\mathsf{T} \frac{\mathbf 1_{2k-1}}{\sqrt{2k-1}}\|_2^2+O_k(m^{-1})
 =\frac{1}{2k-1}\sum_{x\in X}|N_{N_j}(x)|^2+O_k(m^{-1}).
 \end{aligned}
\]
Left-multiplying both sides by $\xi^\mathsf{T}$ for equation \eqref{xiequation} and using
\(\|\xi\|_2=1\), we obtain
 \begin{align}\label{eq:5.04}
 \lambda^2
 ={}&\lambda\xi^\mathsf{T}(J_{2k-1}-I_{2k-1})\xi+n_j\xi^\mathsf{T} J_{2k-1}\xi+\xi^\mathsf{T}\lambda Y(\lambda I_{|X|}-A(G[X]))^{-1}Y^\mathsf{T}\xi.
 \end{align}
By $\|\xi\|_2=1$, \(\|A(G[X])\|_2=O_k(1)\) and $\|Y\|_2=O_k(1)$, we have
$$
|\xi^\mathsf{T} YA(G[X])Y^\mathsf{T}\xi|
 \le \|Y\|_2^2\|A(G[X])\|_2\|\xi\|_2^2
 =O_k(1).
$$
Hence, by the Neumann-series estimate,
 \begin{align}\notag
 \xi^\mathsf{T}\lambda Y(\lambda I_{|X|}-A(G[X]))^{-1}Y^\mathsf{T}\xi
 &=\xi^\mathsf{T}Y\left(I_{|X|}-\frac{A(G[X])}{\lambda}\right)^{-1}
 Y^\mathsf{T}\xi\\ \notag
 &=\xi^\mathsf{T} Y
 \left(I_{|X|}+\frac{A(G[X])}{\lambda}+
 \sum_{i\geq2}\left(\frac{A(G[X])}{\lambda}\right)^i\right)
 Y^\mathsf{T}\xi\\ \notag
 &=\xi^\mathsf{T} YY^\mathsf{T}\xi
 +\lambda^{-1}\xi^\mathsf{T} YA(G[X])Y^\mathsf{T}\xi
 +O_k(\lambda^{-2})\\ \label{eq:5.05}
 &=\frac{1}{2k-1}\sum_{x\in X}|N_{N_j}(x)|^2+O_k(\lambda^{-1}).
 \end{align}
Together with \eqref{eq:5.02}-\eqref{eq:5.05}, one gets
\begin{align}\label{eq:original-8-8}
 \lambda^2
 =(2k-2)\lambda+(2k-1)n_j
 +\frac{1}{2k-1}\sum_{x\in X}|N_{N_j}(x)|^2
 +O_k(\lambda^{-1}).
\end{align}

Recall that the edges of \(G\) has been split into four disjoint classes: the
\(\binom {2k-1}2\) edges inside \(N_j\), the \((2k-1)n_j\) edges between \(N_j\)
and \(\mathcal I_j\), the \(\sum_{x\in X}|N_{N_j}(x)|\) edges between \(N_j\) and \(X\), and
the \(e(G[X])\) edges inside \(X\). 
Hence
\begin{align}\label{eq:original-8-9}
 m=\binom {2k-1}2+(2k-1)n_j+\sum_{x\in X}|N_{N_j}(x)|+e(G[X]).
\end{align}

Recall that $\mathfrak D_k(G)=m-\binom {2k-1}2+2(k-1)\lambda-\lambda^2$. Substituting equation \eqref{eq:original-8-8} and equation \eqref{eq:original-8-9} into this definition, we derive
\begin{align}\label{equationofDkG}
 \mathfrak D_k(G)
 &=m-\binom {2k-1}2-(2k-1)n_j-\frac1{2k-1}\sum_{x\in X}|N_{N_j}(x)|^2
   +O_k(\lambda^{-1})\nonumber\\
 &=e(G[X])+\sum_{x\in X}|N_{N_j}(x)|
   -\frac1{2k-1}\sum_{x\in X}|N_{N_j}(x)|^2
   +O_k(\lambda^{-1})\nonumber\\
 &=e(G[X])+\frac1{2k-1}\sum_{x\in X}|N_{N_j}(x)|(2k-1-|N_{N_j}(x)|)
   +O_k(m^{-1/2}).
\end{align}

Define $h(x)=x(2k-1-x),$ where $0\le x\le 2k-1$.
If \(0\leq x+y\le 2k-1\), then
$
h(x)+h(y)-h(x+y)=2xy\geq0;
$
if \(x+y>2k-1\), then
$
h(x)+h(y)-h(x+y-2k+1)
=2(2k-1-x)(2k-1-y)\geq0.
$

For $0\leq x,y\leq 2k-1$, define
\begin{align*}
x\oplus y=
\begin{cases}
x+y, &\text{if } 0\leq x+y\leq 2k-1,\\
x+y-2k+1, &\text{if } x+y>2k-1.
\end{cases}
\end{align*}
Then $x\oplus y\equiv x+y\pmod{2k-1}$ and $x\oplus y\in\{0,1,\ldots,2k-1\}$. Hence $h(x)+h(y)\geq h(x\oplus y)$. In particular, for all $x,y\in\{1,2,\ldots,2k-2\}$, we have $h(x)+h(y)>h(x\oplus y)$. We call a number \(x\) nontrivial if
\(1\le x\le 2k-2\).

For a finite sequence $d_1,\ldots,d_s\in\{0,1,\ldots,2k-1\}$, if there are at least two nontrivial entries $d_i,d_t$, then a merging operation of this sequence is an operation that replace $d_i$ and $d_t$ by $d_i\oplus d_t$ in this sequence. After this operation, we have
$$
\sum_{r=1}^sd_r\equiv\sum_{r\neq i,t}d_r+(d_i\oplus d_t)\pmod{2k-1},
$$
and the number of nontrivial entries decreases by at least one. Furthermore,
$$
\sum_{r=1}^sh(d_r)>\sum_{r\neq i,t}h(d_r)+h(d_i\oplus d_t).
$$

Let $b\in \{0,\ldots,2k-2\}$ be the residue of $\sum_{r=1}^sd_r$ modulo $2k-1$. Repeatedly merge two nontrivial entries whenever possible. Each merge preserves the residue modulo $2k-1$, decreases the number of nontrivial entries, and strictly decreases the sum of $h$ over the entries. The operation therefore may terminate. If $b=0$, the terminal sequence has no nontrivial entry; otherwise, it has exactly one nontrivial entry, equal to $b$. Since $h(0)=h(2k-1)=0$, the terminal sum is $h(b)$. Consequently, $\sum_{r=1}^sh(d_r)\geq h(b)$, with strict inequality if and only if the original sequence contains at least two nontrivial entries.
\begin{lemma}\label{unique-minimizer}
Let \(e\) and \(d_1,\ldots,d_s\) be integers satisfying
\(e\ge0,\ 0\le d_i\le 2k-1\) for each $1\leq i\leq s$, and \(
e+\sum_i d_i\equiv\ell\pmod {2k-1}
\)
for some \(1\le\ell\le 2k-2\). Then
\(
e+\frac1{2k-1}\sum_i d_i(2k-1-d_i)
\ge\frac{\ell(2k-1-\ell)}{2k-1}.
\)
Equality holds if and only if \(e=0\) and there is exactly one nontrivial entry $d_i=\ell$ for some $1\leq i\leq s$ and $d_t\in\{0,2k-1\}$ for $t\in[s]\setminus\{i\}$. Otherwise, $e+\frac1{2k-1}\sum_i d_i(2k-1-d_i)
\ge\frac{\ell(2k-1-\ell)}{2k-1}+\frac{1}{2k-1}$.
\end{lemma}
\begin{proof}
For convenience, write
\[
 \Phi(e;d_1,\ldots,d_s)
 :=e+\frac1{2k-1}\sum_{i=1}^s h(d_i).
\]
Note that $h(0)=h(2k-1)=0$ and $0\equiv 2k-1\pmod {2k-1}$. Thus entries equal to \(0\) or \(2k-1\) affect neither $\Phi(e;d_1,\ldots,d_s)$ nor the congruence condition $e+\sum_i d_i\equiv\ell\pmod {2k-1}$.

\smallskip
\noindent\emph{Case 1: \(e=0\).} That is, $\sum_{i=1}^s d_i\equiv\ell\pmod {2k-1}$. Note that \(1\le\ell\le 2k-2\). Hence, after a sequence of merging operations of \(d_1,\ldots,d_s\), the terminal sequence must contain exactly
one nontrivial entry \(\ell\), and $\sum_{i=1}^s h(d_i)\ge h(\ell)$.
Consequently,
\[
 \Phi(0;d_1,\ldots,d_s)
 \ge \frac{h(\ell)}{2k-1}
 =\frac{\ell(2k-1-\ell)}{2k-1}.
\]
If the sequence \(d_1,\ldots,d_s\) contains at least two nontrivial entries,
then $\sum_{i=1}^s h(d_i)>h(\ell)$. Hence equality holds only if the sequence \(d_1,\ldots,d_s\) contains exactly one nontrivial entry, equal to \(\ell\). Conversely, such a sequence attains equality.

\smallskip
\noindent\emph{Case 2: \(e>0\).} In this case, we show that the desired inequality is then always strict. Since
\[
 \max_{0\le x\le 2k-1}\frac{h(x)}{2k-1}
 =\max_{0\le x\le 2k-1}\frac{x(2k-1-x)}{2k-1}
 \le \frac {2k-1}4,
\]
if \(e\ge 2k-1\), then
\[
 \Phi(e;d_1,\ldots,d_s)
 \ge e\ge 2k-1>\frac {2k-1}4\ge\frac{h(\ell)}{2k-1}.
\]
It remains to consider \(1\le e<2k-1\).

Let \(b\equiv\sum_{i=1}^sd_i\pmod{2k-1}\in\{0,\ldots,2k-2\}\). Using a sequence of merging operations for the sequence \(d_1,\ldots,d_s\), we have
\(
 \sum_{i=1}^s h(d_i)\ge h(b)
\)
and, by the congruence assumption, we derive
\(
 b\equiv\sum_{i=1}^s d_i
 \equiv\ell-e\pmod {2k-1}.
\)
It follows that
\(
 \Phi(e;d_1,\ldots,d_s)
 \ge e+\frac{h(b)}{2k-1}.
\)

Suppose first that \(1\leq e\le\ell\). Then $2\ell-e\ge\ell>0$. In this case,
\[
 \begin{aligned}
 e+\frac{h(\ell-e)}{2k-1}-\frac{h(\ell)}{2k-1}
 &=
 e+\frac{(\ell-e)(2k-1-\ell+e)-\ell(2k-1-\ell)}{2k-1}
 =\frac{e(2\ell-e)}{2k-1}>0.
 \end{aligned}
\]

Now suppose that \(e>\ell\), and put \(q=e-\ell\). Since \(e<2k-1\), we
have \(1\le q<2k-1\), and the least nonnegative residue of
\(\ell-e=-q\) modulo \(2k-1\) is \(b=2k-1-q\). Therefore
\[
 \begin{aligned}
 e+\frac{h(2k-1-q)}{2k-1}-\frac{h(\ell)}{2k-1}
 &=
 q+\ell+\frac{q(2k-1-q)-\ell(2k-1-\ell)}{2k-1}
 =\frac{q(4k-2-q)+\ell^2}{2k-1}>0.
 \end{aligned}
\]
Therefore, \(\Phi(e;d_1,\ldots,d_s)>\frac{h(\ell)}{2k-1}\) whenever
\(e\geq1\).

If $e+\frac1{2k-1}\sum_i d_i(2k-1-d_i)
>\frac{\ell(2k-1-\ell)}{2k-1}$, that is, $ \frac{(2k-1)e+\sum_{i=1}^s h(d_i)-h(\ell)}{2k-1}>0$, then since $(2k-1)e+\sum_{i=1}^s h(d_i)-h(\ell)$ is a positive integer, we have
$$
\frac{(2k-1)e+\sum_{i=1}^s h(d_i)-h(\ell)}{2k-1}\geq\frac{1}{2k-1},
$$
as desired.
\end{proof}


Now we are ready to prove Theorem~\ref{main-theorem}.
\begin{proof}[\bf Proof of Theorem~\ref{main-theorem}]
Let $G$ be the graph having maximum spectral radius among all the $kK_3$-free graphs with $m$ edges and no isolated vertex. Then by the discussion in Section~4 and Lemma~\ref{bounded-multiplicity-of-nonprincipal-types}, $G$ has a vertex partition $V(G)=N_j\sqcup \mathcal I_j\sqcup X$ such that $G[N_j]\cong K_{2k-1}$, $\mathcal I_j$ is independent, $N(u)=N_j$ for $u\in \mathcal I_j$, $|X|=O_k(1)$ and $E(\mathcal I_j,X)=\emptyset$. By equation \eqref{eq:original-8-9} and \(m-\binom {2k-1}2\equiv\ell\pmod {2k-1}\), we have
\[
e(G[X])+\sum_{x\in X}|N_{N_j}(x)|\equiv\ell\pmod {2k-1}.
\]
Applying Lemma \ref{unique-minimizer} with \(e=e(G[X])\) and the sequence \(|N_{N_j}(x)|,\,x\in X\), we have $e(G[X])+\frac1{2k-1}\sum_{x\in X}|N_{N_j}(x)|(2k-1-|N_{N_j}(x)|)\ge D_\ell$. If equality does not hold, then
\[
e(G[X])+\frac1{2k-1}\sum_{x\in X}|N_{N_j}(x)|(2k-1-|N_{N_j}(x)|)
\ge D_\ell+\frac1{2k-1}.
\]
Equation \eqref{equationofDkG} and the above inequality imply
\begin{align}\label{eq:original-10-4}
\mathfrak D_k(G)
\ge D_\ell+\frac1{2k-1}+O_k(m^{-1/2}).
\end{align}
On the other hand, by \eqref{boundofdelta} and \eqref{bound DkG}, we have \(\mathfrak D_k(G)\leq\delta_\ell<D_\ell\). This is a contradiction. Hence equality must
hold in Lemma \ref{unique-minimizer}. We conclude that \(e(G[X])=0\) and there is exactly one vertex $x\in X$ with \(|N_{N_j}(x)|=\ell\) and $|N_{N_j}(y)|\in\{0,2k-1\}$ for every $y\in X\setminus \{x\}$.

If $|N_{N_j}(y)|=0$ for some $y\in X\setminus \{x\}$, then $y$ is isolated because \(G[X]\)
has no edges and \(E(\mathcal I_j,X)=\emptyset\), which contradicts that $G$ has no isolated vertex. Thus $|N_{N_j}(y)|=2k-1$ for all $y\in X\setminus \{x\}$, then $N_G(y)=N_j$. Thus \(G\) consists of the clique \(G[N_j]\cong K_{2k-1}\), an independent set satisfying that the neighborhood of every vertex in this set is \(N_j\), and one vertex adjacent to precisely \(\ell\) vertices of \(N_j\). 
Thus, $G\cong T_{k,m}^{(\ell)}$. Therefore, for an arbitrary $kK_3$-free graph $G$ with $m$ edges attaining the maximum spectral radius, we have $G\cong T_{k,m}^{(\ell)}\cup sK_1$ for some $s\geq0$.
\end{proof}

Combining Theorem \ref{main-theorem} with the divisible case of Das and Yamini gives a
single statement for every sufficiently large size.

\begin{corollary}
Fix \(k\ge2\), let \(m\) be sufficiently large.
\begin{enumerate}[label=\textup{(\roman*)}]
\item
  If \(m\equiv0\pmod {2k-1}\), then
\(
\operatorname{spex}(m,kK_3)
=(k-1)+\sqrt{m-k(k-1)},
\)
and the unique extremal graph up to isolated vertices is
\(
K_{2k-1}\vee\left(\frac m{2k-1}-(k-1)\right)K_1.
\)
\item
  If \(m\equiv\ell\pmod {2k-1}\) with \(1\le\ell\le 2k-2\). Then
\(
\operatorname{spex}(m,kK_3)
=\lambda\bigl(T_{k,m}^{(\ell)}\bigr),
\)
the largest real zero of the quartic \eqref{polynomialofextremalgraph}, and the unique extremal
graph up to isolated vertices is \(T_{k,m}^{(\ell)}\).
\end{enumerate}
\end{corollary}
\section*{Statements and Declarations}

\noindent\textbf{Funding}
Shuchao Li was supported by the National Natural Science Foundation of China (Grant Nos. 12571365, 121971190). Yuantian Yu was supported by the Early-Career Talent Training Fund of Jiangxi Province (Grant No. 20262BEJ730309) and the Doctoral Research Startup Fund of East China University of Technology (Grant No. DHBK2026038).

\noindent\textbf{Competing interests}
The authors declare that they have no competing interests.

\noindent\textbf{Data availability}
No data were generated or analysed during the current study.


\noindent\textbf{Use of AI tools}
During the preparation of this manuscript, the authors used Doubao for language polishing and structural organization. The authors reviewed and edited all content and take full responsibility for the final manuscript.

\end{document}